\documentclass[11pt,a4paper]{article} 

\usepackage[usenames,dvipsnames]{color}
\usepackage{tikz}
\usetikzlibrary{shapes}
\usepackage{lscape}

\usepackage[utf8]{inputenc}
\usepackage[T1]{fontenc}

\usepackage[english]{babel}

\usepackage{a4wide}

\usepackage{amssymb}
\usepackage{amsthm,amsmath}
\usepackage{amsfonts}
\usepackage{latexsym}
\usepackage{yhmath}

\usepackage{url} 
\usepackage[all]{xy} 
\usepackage{enumerate}
\usepackage{enumitem}

\usepackage[figuresright]{rotating}

\usepackage{graphicx}

\theoremstyle{plain}
\newtheorem{thm}{Theorem}
\newtheorem*{thm*}{Theorem}
\newtheorem{lem}{Lemma}[section]

\newtheorem{prop}[lem]{Proposition}

\theoremstyle{definition}
\newtheorem{df}[lem]{Definition}
\newtheorem{rem}[lem]{Remark}
\newtheorem{ex}[lem]{Example}

\newcommand{\bbR}{\mathbb{R}}

\newcommand{\bbZ}{\mathbb{Z}}

\newcommand{\bbC}{\mathbb{C}}
\newcommand{\bbK}{\mathbb{K}}
\newcommand{\bbH}{\mathbb{H}}
\newcommand{\bbO}{\mathbb{O}}

\newcommand{\R}{\mathbb{R}}
\newcommand{\Z}{\mathbb{Z}}
\newcommand{\C}{\mathbb{C}}

\newcommand{\bigslant}[2]{{\raisebox{.2em}{$#1$}\left/\raisebox{-.2em}{$#2$}\right.}}

\makeatletter

\@addtoreset{equation}{section}
\makeatother

\begin{document}

\title{Bott type periodicity for the higher octonions}

\date{}

\author{Marie Kreusch}

\maketitle

\begin{abstract}
We study the series of complex nonassociative algebras $\bbO_{n}$ 
and real nonassociative algebras $\bbO_{p,q}$
 introduced in~\cite{MGO2011}.
These algebras generalize the classical algebras of octonions and Clifford algebras.
The algebras $\bbO_{n}$ and $\bbO_{p,q}$ with $p+q=n$
have a natural $\Z_2^n$-grading, and they are
characterized by cubic forms over the field $\Z_2$.
We establish a periodicity for the algebras~$\bbO_{n}$ and $\bbO_{p,q}$ similar to 
that of the Clifford algebras $\mathrm{Cl}_{n}$ and~$\mathrm{Cl}_{p,q}$.
\end{abstract}


\section{Introduction}

A series of noncommutative and nonassociative algebras $\bbO_{p,q}$ over $\R$
and their complexification~$\bbO_{n}$ (where $n=p+q$) 
was recently introduced~\cite{MGO2011} and studied in~\cite{LMGO2011,MGO2013,KMG2014}.
The algebras~$\bbO_{n}$ and $\bbO_{p,q}$ generalize the classical algebras of octonions and split octonions
in the same way as the Clifford algebras generalize the quaternions. 
Note that the algebra of octonions $\bbO$ appears in the series as $\bbO_{0,3}$,
whereas the algebra of split octonions is isomorphic to~$\bbO_{3,0}$,~$\bbO_{1,2}$ and $\bbO_{2,1}$.
The properties of the algebras $\bbO_{n}$ and $\bbO_{p,q}$ are very different from
those of the classical Cayley-Dickson algebras.  

The series of algebras $\bbO_{n}$ and $\bbO_{p,q}$ can be illustrated by the following diagram. 
\medskip

\begin{figure}[htbp]
\begin{tikzpicture}[every text node part/.style={align=center}, scale=2.7]
  \node[draw, rounded corners=3pt] (CDA) at (5,0) {\footnotesize{Cayley-Dickson} \\  \footnotesize{algebras}};  
  \node[draw=white] (R) at (0,0) {$\bbR$};
    \node[draw=white] (C) at (1,0) {$\bbC$};
\draw [>=stealth,->] (R) -- (C);
    \node[draw=white] (H) at (2,0) {$\bbH$};
\draw [>=stealth,->] (C) -- (H);
    \node[draw=white] (O) at (3,0) {$\bbO$};
\draw [>=stealth,->] (H) -- (O);
    \node[draw=white] (S) at (4,0) {  $\quad $ $\ldots$};
\draw [>=stealth,->] (O) -- (S);
    \node[draw=white] (Cl3) at (2,-0.5) {$\mathrm{Cl}_{0,3}$};
    \node[draw=white] (Cl4) at (2,-1) {$\mathrm{Cl}_{0,4}$};
    \node[draw=white] (Cldot) at (2,-1.5) {$\vdots$};
\draw[->,>=latex] (Cl3) -- (Cl4);
\draw[->,>=latex] (H) -- (Cl3);
\draw[->,>=latex] (Cl4) -- (Cldot);
  \node[draw, rounded corners=3pt] (CDA) at (2,-2) {\footnotesize{Clifford} \\  \footnotesize{algebras}};  
   \node[draw=white] (O4) at (3,-0.5) {$\bbO_{0,4}$};
    \node[draw=white] (O5) at (3,-1) {$\bbO_{0,5}$};
    \node[draw=white] (Odot) at (3,-1.5) {$\vdots$};
\draw[->,>=latex] (O4) -- (O5);
\draw[->,>=latex] (O) -- (O4);
\draw[->,>=latex] (O5) -- (Odot);
  \node[draw, rounded corners=3pt] (CDA) at (3,-2) { \footnotesize{algebras} \\ $\bbO_{p,q}$};  
\end{tikzpicture}
\label{diagramalgebras}\caption{Families of $\Z_2^n$-graded algebras}
\end{figure}
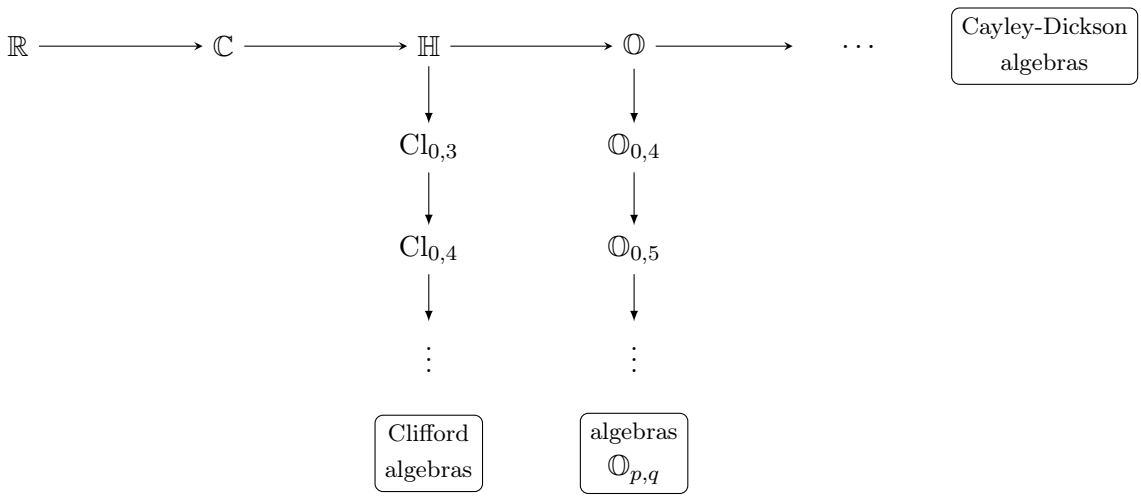

\medskip

The complex algebras $\bbO_{n}$ and especially the real algebras $\bbO_{0,n}$ 
have applications to the classical Hurwitz problem of sum of square identities 
and related problems; see~\cite{LMGO2011,MGO2013}. 
An application of $\bbO_{n}$ to additive combinatorics was suggested in~\cite{SV}.

The idea to understand the classical algebra of the octonions $\bbO$ as a graded algebra
was suggested in~\cite{Eld}, where, in particular, a~$\Z_2^3$-grading was considered\footnote{
Throughout the paper we denote by~$\Z_2$ the quotient $\Z/2\Z$
understood as abelian group, and also as a field of two elements $\{0,1\}$.}.
In~\cite{AM1999}, the algebra $\bbO$ was understood as a twisted group algebra
over~$\Z_2^3$ which has a graded commutative and graded associative structure.

The algebras $\bbO_{n}$ and $\bbO_{p,q}$ are graded algebras over
the abelian group $\Z_2^n$; the algebra $\bbO_{0,3}$ is isomorphic to $\bbO$.
Moreover, they are characterized by a {\it cubic form}
$$
\alpha:\Z_2^n\to\Z_2,
$$  
where $\Z_2^n$ is understood as a vector space of dimension $n$ over 
the field~$\Z_2$ of two elements; see~\cite{MGO2011}.
This is the main property of the algebras $\bbO_{n}$ and $\bbO_{p,q}$
that distinguish them from other series of algebras generalizing the octonions,
such as the Cayley-Dickson algebras.

The problem of classification of the real algebras $\bbO_{p,q}$ 
with {\it fixed} $n=p+q$ depending on the signature $(p,q)$, 
was formulated in~\cite{MGO2011}.
This problem was solved in~\cite{KMG2014},
the result is as follows.
The classification table of $\bbO_{p,q}$ for $p$ and $q\not=0$, coincides with the well known 
table of the real Clifford algebras; the algebras $\bbO_{0,n}$ and $\bbO_{n,0}$ are exceptional. 

The present paper answers the following problem: 
how do the algebras $\bbO_{n}$ and $\bbO_{p,q}$ with~$p+q=n$
depend on the parameter $n$?
Similarity with the Clifford algebras allows one to expect properties of periodicity,
in particular it is natural to look for analogs of so-called {\it Bott periodicity}; see~\cite{Baez2002}. 

We consider the problem of periodicity in the complex and in the real cases separately.
We establish a periodicity modulo $4$.
In the complex case, we
link together the algebras~$\bbO_{n}$ and~$\bbO_{n+4}$.
Note that for the complex Clifford algebras there is a simple periodicity modulo~$2$.
In the real case,
for the algebras $\bbO_{p,q}$ (provided $p>0$ and $q>0$) we establish 
two different results about periodicity modulo $4$. 
The situation for the exceptional algebras $\bbO_{0,n}$ and $\bbO_{n,0}$ is different. 
The results are compared to the well-known results for the Clifford algebras $\mathrm{Cl}_{p,q}$.

\section{The algebras $\bbO_{n}$ and $\bbO_{p,q}$}  \label{Sectiondefinition}
In this section, we recall definitions of the complex algebras $\bbO_{n}$ and of the real algebras $\bbO_{p,q}$,
as twisted group algebras over~$\Z_2^n$ characterized by a cubic form. 
We then give an equivalent definition of  $\bbO_{n}$ and $\bbO_{p,q}$ in term of generators and relations. 
Finally, we recall the main results of classification from \cite{KMG2014}. 
 
\subsection{$\bbO_{n}$ and $\bbO_{p,q}$ as twisted group algebras over $\bbZ_2^n$}
We denote by $\bbK$ the ground field assumed to be $\bbR$ or $\bbC$.
Let  $f$ be an arbitrary function in two arguments
$$
f:\bbZ_2^n \times \bbZ_2^n\to\bbZ_2.
$$
The {\it twisted group algebra} $\mathcal{A}= (\bbK\left[  \bbZ_2^n \right],f)$
(for more detail see~\cite{Conlon1964} and \cite{MGO2011}) 
is defined as the $2^n$-dimensional vector space with the basis 
$\{u_x, x\in \bbZ_2^n\}$, and equipped with the product
\begin{equation*}
\label{twistedproduct}
u_x \cdot u_y = (-1)^{f(x,y)} u_{x+y},
\end{equation*}
for all $x,y\in \bbZ_2^n$.

\begin{ex}
(a)
Recall that the real {\it Clifford algebra} denoted by $\mathrm{Cl}_{p,q}$
is the associative algebra with $n=p+q$ generators $v_1,\ldots,v_n$
satisfying the relations
\begin{equation}
\label{CliFGen}
\begin{array}{rcl}
v_i^2 &=& \left\{
\begin{array}{rl}
1, & 1 \leq i \leq p,\\
-1, & p+1 \leq i \leq p+q,
\end{array}
\right.
\\[10pt]
v_i \cdot v_j&=&-v_j \cdot  v_i, 
\end{array}
\end{equation}
for all $i \neq j  \leq n$.
Obviously, $\dim\mathrm{Cl}_{p,q}=2^n$, and a natural basis is
$$
\left\{
v_{i_1}\cdots{}v_{i_k}\vert{}\;1\leq{}i_1<\ldots<i_k\leq{}n
\right\}.
$$

The (real) algebra of quaternions $\bbH$
($\,\simeq\mathrm{Cl}_{0,2}$), and more generally every complex or real Clifford algebra 
with $n$ generators can be realized as twisted group algebras over $\bbZ_2^n$; see \cite{AM2002}.
Denote by $x=(x_1, \ldots,x_n)$ and $y=(y_1, \ldots,y_n)$ the elements in $\bbZ_2^n$
(where the components $x_i$ and $y_i$ are equal to $0$ or $1$) and defined two functions given by 
\begin{equation*}
\label{fclassiccomplex}
f _{\mathrm{Cl}_{n}} \left( x,y \right):= 
\sum_{1 \leq i \leq j \leq n} x_i y_j , 
\end{equation*}
\begin{equation*}
\label{fclassic}
f _{\mathrm{Cl}_{p,q}} \left( x,y \right):= f _{\mathrm{Cl}_{n}}  \left( x,y \right)+ \sum_{1\leq{}i\leq{}p}x_i y_i \quad \quad (n=p+q).
\end{equation*}
Then the defined twisted group algebras are isomorphic to $\mathrm{Cl}_{n}$
in the complex case, and to~$\mathrm{Cl}_{p,q}$ in the real case.
In particular, $f _{\bbH} \left( x,y \right)=x_1y_1+x_1y_2+x_2y_2$ corresponds to the algebra
of quaternions.

(b)
The (real) algebra of octonions $\bbO$ is a twisted group algebra over $\bbZ_2^3$; see \cite{AM1999}.
The twisting function is cubic:
\begin{equation*}
\label{fclassic2}
f _{\bbO} \left( x,y \right)=
(x_1 x_2 y_3 +  x_1 y_2 x_3 + y_1 x_2 x_3 ) + \sum_{1 \leq i \leq j \leq 3} x_i y_j   .
\end{equation*}
\end{ex}

The next definition is the main object of the present paper.

\begin{df}
\label{definition}
\cite{MGO2011}
The \emph{complex algebra $\bbO_{n}$}
and the \emph{real algebra $\bbO_{p,q}$} with $p+q=n\geq 3$ are the twisted group algebras 
with the twisting functions
\begin{eqnarray*}
\label{fO}
f_{\bbO_{n}} \left( x,y \right)&=& 
\sum_{1 \leq i < j < k \leq n} (x_i x_j y_k +  x_i y_j x_k + y_i x_j x_k ) + 
\sum_{1 \leq i \leq j \leq n} x_i y_j , \\[4pt]
\label{fOpq}
f _{\bbO_{p,q}} \left( x,y \right)&=& 
f _{\bbO_n} \left( x,y \right)
+ \sum_{1\leq{}i\leq{}p}x_i y_i, 
\end{eqnarray*}
respectively.
Note that the element $1:= u_{(0, \ldots , 0)}$ is the {\it unit} of the algebra. 
\end{df}

The real algebra $\bbO_{0,3}$ is nothing but the classical algebra $\bbO$ of octonions.

\begin{df}
For both series of algebras $\mathrm{Cl}_{p,q}$ and $\bbO_{p,q}$
the index $(p,q)$ is called the {\it signature}, and throughout the paper we assume $p+q=n$.  
\end{df}

\subsection{Graded-commutative and graded-associative algebras}

Every twisted algebra $(\bbK[\bbZ^n_2],f)$ is a graded algebra over the group~$\bbZ_2^n$. 
In general, a twisted group algebra is neither commutative nor associative. 
The defect of commutativity and associativity is
measured by a symmetric function $\beta:\bbZ_2^n \times \bbZ_2^n\to\bbZ_2$, 
and a function $\phi:\bbZ_2^n \times \bbZ_2^n \times \bbZ_2^n \to\bbZ_2$, respectively:
\begin{eqnarray} 
\label{comass1}
u_x \cdot u_y &=& (-1)^{\beta(x,y)}\; u_y \cdot u_x \label{comass1} ,\\[4pt]
u_x \cdot ( u_y \cdot u_z)&  =& (-1)^{\phi(x,y,z)} \;(u_x \cdot  u_y) \cdot u_z  \label{comass2},
\end{eqnarray}
where $\beta $ and $\phi$ are given by
\begin{eqnarray} 
\label{BetEq}
\beta(x,y) &=& f(x,y)+f(y,x) ,\\[4pt]
\label{PhiEq}
\phi(x,y,z)&  =& f(x,y)+f(x,y+z)+f(x+y,z)+f(y,z).
\end{eqnarray}
Note that the second formula reads $\phi=\delta{}f$ where $\delta$ is the coboundary operator.
Therefore, the function $\phi$ is a trivial $3$-cocycle on $\bbZ_2^n$ with coefficients in $\Z_2$.

Algebras satisfying the relations (\ref{comass1}) and (\ref{comass2}) are called
{\it graded-commutative} and {\it graded-associative}, respectively.
In particular, the algebras $\bbO_{n}$ and $\bbO_{p,q}$ are 
graded-commutative and graded-associative.

\begin{rem}
Note also that the algebras $\bbO_{n}$ and $\bbO_{p,q}$ are {\it graded-alternative}, 
 i.e., 
 $$
 u\cdot(u\cdot v)=u^2\cdot v,
 $$ 
 for all homogeneous elements~$u,v$.
\end{rem}

\subsection{The algebras $\bbO_{n}$ and $\bbO_{p,q}$: generators and relations} \label{abstract algebras}

We give here another, equivalent definition of the algebras $\bbO_{n}$ and $\bbO_{p,q}$ 
with the help of generators and relations. 
 
Let denote the basis elements of the abelian group $\bbZ_2^n$, 
\begin{equation}
\label{Gener}
e_i=(0,\ldots ,0,1,0,\ldots,0)
\end{equation}
 where $1$ stands at the $i^{th}$ position. 
 The homogeneous elements $u_i := u_{e_i}$, $1\leq i \leq n$ 
 form a set of generators of the group algebra $\bbK[\bbZ^n_2]$.
 
Let $u= u_{i_1} \cdots  u_{i_k}$ 
 be a monomial in the generators, the {\it degree} of $u$ is the element of $\bbZ^n_2$ given by
  $$
 \bar u := \bar u_{i_1} + \ldots +   \bar u_{i_k},
 $$
where the degree of the generator $u_{i}$ is $ \bar u_{i}={e_i}$. 
The monomials form a basis of the group algebra~$\bbK[\bbZ^n_2]$.

It was shown in~\cite{MGO2011} that 
the exists a {\it unique} trilinear form $\phi:\Z_2^n\times\Z_2^n\times\Z_2^n\to\Z_2$
such that 
\begin{equation}
\label{Uniq}
\phi(e_i,e_j,e_k)=1,
\end{equation}
for all distinct $i, j$ and $k$ in $\{1, \ldots ,n\}$. 
The algebras $\bbO_{p,q}$ can be equivalently defined
as follows.

\begin{df}
\begin{enumerate}
\item[(a)]
The algebra $\bbO_{p,q}$ is the unique real unital algebra, 
generated by $n$ elements $ u_1, \ldots , u_n $ $(p+q=n)$,
subject to the relations
\begin{eqnarray*}
\label{genrel}
u_i^2 &=& \left\{
\begin{array}{rl}
1, & 1 \leq i \leq p,\\
-1, & p+1 \leq i \leq p+q,
\end{array}
\right.
\\[5pt]
u_i \cdot u_j&=&-u_j \cdot  u_i, 
\end{eqnarray*}
for all $i \neq j  \leq n$,
together with the graded associativity
\begin{equation*}
\label{MonRel}
u\cdot(v\cdot w)=(-1)^{\phi(\bar u, \bar v, \bar w)}(u\cdot v)\cdot w,
\end{equation*}
where $u,v,w$ are monomials, and where $\phi$ is the unique trilinear form satisfying (\ref{Uniq}). 

\item[(b)]
The algebra $\bbO_{n}$ is the complexification of $\bbO_{p,q}$,
its generators satisfy
the same relations.

\end{enumerate}
\end{df}

Clearly, the complexifications of $\bbO_{p,q}$ and $\bbO_{p',q'}$ with $p+q=p'+q'=n$
are isomorphic.

The following observation is important.

\begin{rem}
The trilinear form $\phi$ is {\it symmetric} in three arguments, i.e.,
\begin{equation}
\label{Sic}
\phi(x,y,z)=\phi(x,z,y)=\cdots=\phi(z,y,x),
\end{equation}
for all $x,y,z\in\Z_2^n$. 
\end{rem}

\subsection{Classification of $\bbO_{p,q}$}

Classification of $\bbO_{p,q}$ as $\Z_2^n$-graded algebras was obtained in \cite{KMG2014}. 
It consists in the list of isomorphisms between these algebras 
that preserve the structure of $\bbZ_2^n$-graded algebra
(i.e., isomorphisms sending homogeneous elements into homogeneous)
are as follows. 

\begin{prop}
\label{thmIso}
If $pq\not=0$, then there are the following isomorphisms of graded algebras:
\begin{enumerate}
\item[(i)] 
$\bbO_{p,q}\simeq \bbO_{q,p}\;$;
\item[(ii)] 
$\bbO_{p,q+4}\simeq \bbO_{p+4,q}\;$;
\item[(iii)] 
For $n\geq 5$, the algebras $\bbO_{n,0}$ and $\bbO_{0,n}$ are not isomorphic,
and are not isomorphic to any other algebras $\bbO_{p,q}$ with $p+q=n$.
\end{enumerate}
All the isomorphisms between the algebras $\bbO_{p,q}$ are as above.
\end{prop}

Note that, apart for the algebras $\bbO_{n,0}$ and $\bbO_{0,n}$ which are exceptional,
the above classification is quite similar to the classification of $\mathrm{Cl}_{p,q}$.

A kind of degeneracy occurs in the small dimensions, since for $n=3$, one has :  
\[
\bbO_{3,0} \simeq \bbO_{2,1}\simeq \bbO_{1,2} \not\simeq  \bbO_{0,3},
\]
and for 
 $n= 4$, one has : 
\[
\bbO_{4,0} \simeq \bbO_{2,2} \not\simeq \bbO_{1,3} \simeq  \bbO_{3,1} \not\simeq \bbO_{0,4} .
\]

Let us also mention the following criterion of simplicity from~\cite{MGO2011}.

\begin{prop}
The algebra $\bbO_{p,q}$  is simple if and only if $p+q \not\equiv 0 \mod 4$,
or $p+q \equiv 0 \mod 4$ and $p,q$ are odd. 
\end{prop}

\section{The generating cubic form} \label{abstract algebras}

We will be needing  a theory, developed in \cite{MGO2011}, 
about a class of twisted algebras over $\bbZ_2^n$
that are characterized by a cubic form;
this is the case for the algebras $\bbO_{n}$ and $\bbO_{p,q}$.
The structure of twisted group algebras that can be equipped with 
generating function is much simpler than that
of arbitrary  twisted group algebras.
Note that this class contains such interesting algebras as the code loops~\cite{Gri} (see~\cite{MGO2011}),
whereas the Cayley-Dickson algebras higher that the octonions
do not belong to this class.

\subsection{The notion of generating function} 
\begin{df}\label{defalpha}
Given a twisted group algebra $(\bbK\left[  \bbZ_2^n \right],f)$,
a function $\alpha: \bbZ_2^n \longrightarrow \bbZ_2$ is called
a \emph{generating function} if
$$
\begin{array}{llcl}
(i)&f(x,x)&=&\alpha(x),\\[6pt]
(ii)&\beta( x, y )  &=& \alpha (x+y) + \alpha (x) + \alpha (y) , \\[6pt]
(iii)&\phi( x, y, z )  &=& \alpha (x+y+z ) + \alpha (x+y) + \alpha (x+z) +  \alpha (y+z) + \alpha (x) + \alpha (y) + \alpha(z),
\end{array}
$$
where $x,y,z\in\bbZ_2^n$ and where $\beta$ and $\phi$ are as in~(\ref{BetEq}) and (\ref{PhiEq}).
\end{df}

The following statements were proved in \cite{MGO2011}.

\begin{enumerate} 
\item
\label{statement1}
A twisted group algebra $(\bbK\left[  \bbZ_2^n \right],f)$
has a generating function if and only if the function~$\phi:=\delta{}f$ is 
symmetric as in \eqref{Sic}.

\item
\label{statement2}
The generating function $\alpha$ is a polynomial on $\Z_2^n$ of degree $\leq3$.

\item
\label{statement3}
Given any polynomial $\alpha$ on $\Z_2^n$ of degree $\leq3$, 
there exists a unique (up to isomorphism) twisted group algebra~$(\bbK\left[  \bbZ_2^n \right],f)$ 
having $\alpha$ as a generating function.
\end{enumerate}

It follows that if a twisted group algebra has a generating function then it is
completely characterized by this function.

\subsection{Cubic forms on $\bbZ^n_2$ and twisted group algebras} 
Every cubic form $\alpha:\bbZ^n_2\to\Z_2$ is as follows:
\begin{equation}
\label{AlForm}
\alpha(x)=
\sum_{1\leq{}i\leq{}j\leq{}k\leq{}n}
A_{ijk}\,x_ix_jx_k,
\end{equation}
where the coefficients $A_{ijk}=0$ or $1$.
Note that over $\Z_2$ one has $x_i^3=x_i^2=x_i$, and
therefore, every polynomial of degree $\leq3$ is a homogeneous cubic form.
The general theory of such cubic forms is not developed, and the classification
is unknown; see~\cite{Hou}. 

One can define a twisting function $f_\alpha$
associated with a cubic form $\alpha$ according to the following explicit procedure.
To every monomial one associates:
\begin{equation}
\label{ExplEq}
\begin{array}{rcl}
x_ix_jx_k&\longmapsto&x_ix_jy_k+x_iy_jx_k+y_ix_jx_k,\\[4pt]
x_ix_j&\longmapsto&x_iy_j,\\[4pt]
x_i&\longmapsto&x_iy_i,
\end{array}
\end{equation}
where $1 \leq i<{}j<{}k \leq n$.
Then one extends the above map to the cubic polynomial $\alpha$ by linearity in monomials.

\begin{prop}
Given a cubic function $\alpha$, the corresponding function $f_\alpha$ satisfies
Properties~\ref{statement1}, \ref{statement2} and \ref{statement3} above.
\end{prop}

\begin{rem}
Note that the procedure~(\ref{ExplEq}) is not the unique way to associate
the twisting function to a cubic form.
However, any other procedure would lead to an isomorphic algebra; see~\cite{MGO2011}.
\end{rem}

\subsection{The generating functions of $\bbO_{n}$ and  $\bbO_{p,q}$} 

The algebras $\bbO_{n}$ and  $\bbO_{p,q}$ have the following generating functions:
\begin{eqnarray*}
\label{Genn}
\alpha_{n} (x) &= &\sum_{1 \leq i < j < k \leq n}  x_i x_j x_k   + 
\sum_{1 \leq i  <j \leq n} x_i x_j + \sum_{1 \leq i \leq n }x_i,\\[4pt]
\label{Genpq}
\alpha_{p,q} (x) &=&\alpha_{n} (x) + \sum_{1\leq{}i\leq{}p} x_i . 
\end{eqnarray*}
The cubic form $\alpha_{n}$ of $\bbO_n$ is invariant under the action of
the group of permutations of the coordinates.
The value $\alpha_{n}(x)$ depends only on the weight (i.e. the number of nonzero components) 
of $x$.
More precisely, $\alpha_{n}(x)=0$ if and only if the weight of $x$ is congruent to $0$ modulo~$4$.  

\begin{rem}
In the case of Clifford algebra $\mathrm{Cl}_n$ or $\mathrm{Cl}_{p,q}$, the generating 
functions are the following quadratic form:
\begin{eqnarray*}
\alpha_{\mathrm{Cl}_n} (x) &=& 
\sum_{1 \leq i  < j \leq n} x_i x_j + \sum_{1 \leq i \leq n }x_i,\\[4pt]
\alpha_{\mathrm{Cl}_{p,q}} (x) &=& 
\alpha_{\mathrm{Cl}_n} (x) + \sum_{1 \leq i \leq p }x_i.
\end{eqnarray*}
This was also noticed in~\cite{MGO2011}.
\end{rem}

\subsection{The problem of equivalence} 

\begin{df}
Two cubic forms $\alpha$ and $\alpha'$ on $\bbZ^n_2$ are \emph{equivalent} 
if there exists a linear transformation $G\in\mathrm{GL}_n(\bbZ_2)$ such that
\[  \alpha(x) = \alpha'(G x).  \]
\end{df}

The main method that we use to establish isomorphisms between
twisted group algebras with generating functions is based on the fact that
two equivalent cubic forms give rise to isomorphic algebras. 
More precisely, one has the following statement 
which is an obvious corollary of the uniqueness of the generating function.

\begin{prop}
Given two twisted group algebras,~$(\bbK\left[  \bbZ_2^n \right],f)$ 
and~$(\bbK\left[  \bbZ_2^n \right],f')$ with equivalent generating functions $\alpha$ and $\alpha'$,
then these algebras are isomorphic as $\bbZ_2^n$-graded algebras.
\end{prop}

Let us mention that the general problem of classification of cubic forms on $\bbZ^n_2$
is an old open problem; see~\cite{Hou}.

\section {Periodicity} \label{Sectionclassification}

In this section, we formulate our main results in comparison with the classical results
about the Clifford algebras.
The proofs will be given in Section~\ref{subsectionisomorphisms}.
The main difference between the periodicity theorems that we obtain and the classical ones
is that all the periodicities for the algebras $ \bbO_{n} $ and $ \bbO_{p,q} $
are modulo $4$, whereas in the case of Clifford algebras the
simplest way to formulate the periodicity properties is
modulo $2$.

\subsection{Statement of the main theorem in the complex case} \label{resultcomplex}
Let us recall that for the complex Clifford algebras, one has the following simple
statement: 
$$
\mathrm{Cl}_{n+2}\simeq\mathrm{Cl}_n\otimes_\C\mathrm{Cl}_2.
$$
Note also that $\mathrm{Cl}_2$ is isomorphic to the algebra of complex $2\times2$-matrices.
Our first goal is to establish a similar result for the algebras $ \bbO_{n} $.

\begin{thm}
\label{thmcomplex}
If $n \geq 3 $, there is an isomorphism:
\begin{equation}
\label{FirstProd}
\bbO_{n+4}
\simeq 
\bbO_{n} \otimes_{\C^2}  \bbO_{5} .   
\end{equation}
\end{thm}

Let us explain the tensor product over
 the algebra $\C^2\simeq\C\oplus\C$.
We fix a subalgebra of both above algebras, $\bbO_n$ and $\bbO_5$, isomorphic to $\C^2$: 
$$
\C^2\subset\bbO_{n}
\qquad\hbox{and}\qquad
\C^2\subset\bbO_{5}.
$$
For this end, we chose the generator $u_1=u_{e_1}$, see formula~(\ref{Gener}),
and consider the subalgebra~$\bbC^2$ with basis $\{1,u_1\}$.
Abusing the notation, we call the same name the generator $u_1\in\bbO_{n} $ and
the generator $u_1\in\bbO_{5} $.
We then identify the right action of~$\C^2$ on $\bbO_{n}$ and the left action of~$\C^2$ on $\bbO_{5}$.
In other words, we introduce the ideal $I$ of the algebra $\bbO_n \otimes_{\bbC} \bbO_5$ 
generated by the elements
$$
\left\{ a \cdot  u_1 \otimes_{\bbC}  b -  a \otimes_{\bbC} u_1 \cdot b \hspace{0.2cm}  | 
\hspace{0.1cm} a \in\bbO_{n}, b \in \bbO_{5}   \right\}.
$$ 
The tensor product~(\ref{FirstProd}) is then defined as the quotient by this ideal:
$$
\bbO_{n} \otimes_{\C^2}  \bbO_{5}:=
 \bigslant{ \left( \bbO_{n} \otimes_{\C}  \bbO_{5}   \right)} {I} .   
 $$
Let us stress on the fact that the defined tensor product is the common notion
of tensor product over a subalgebra.

\subsection{Statement of the main theorem in the real case} \label{resultreal}
In the real case, the result is different in the case of algebras
$\bbO_{p,q}$, where $p,q >0$, and in the case of the exceptional algebras 
$\bbO_{n,0}$ and $\bbO_{0,n}$. 
\begin{thm}
\label{thmreal}
If $n=p+q \geq 3$ and if $p>0$ and $q >0$ (except for $(p,q)=(1,4)$ and $(p,q)=(4,1)$) then there are the following isomorphisms of graded algebras:
\begin{eqnarray*}
\bbO_{0,n+4} &\simeq &  \bbO_{0,n} \otimes_{\C}  \bbO_{0,5} ;  \\[4pt]
\bbO_{n+4,0} &\simeq&  \bbO_{n,0} \otimes_{\bbR^2}  \bbO_{5,0} ;\\[4pt]
\bbO_{p+2,q+2} &\simeq&   \bbO_{p,q} \otimes_{\bbC}  \bbO_{2,3} \;
\simeq \bbO_{p,q} \otimes_{\bbR^2}  \bbO_{3,2} .
\end{eqnarray*}
\end{thm}

To describe the above tensor products, chose the generator $u_1$ and
consider a $2$-dimensional subalgebra with basis $\{1,u_1\}$.
If $p=0$, then $u_1^2= -1$, and therefore this (real) subalgebra is isomorphic to $\C$.
If $q=0$, then $u_1^2= 1$, and the chosen subalgebra is isomorphic to $\bbR^2\simeq\bbR \oplus \bbR$.
If $p>0$ and $q>0$, then one can chose $u_1^2= -1$ or $u_1^2= 1$,
the corresponding subalgebra is then isomorphic to $\C$ or $\R^2$, respectively.

In order to compare the above theorem with the classical results for the Clifford algebras,
we recall following periodicities:
\begin{eqnarray*}
 \mathrm{Cl}_{p+2,q} &\simeq&  \mathrm{Cl}_{q,p}\otimes_{\bbR}\mathrm{Cl}_{2,0}  ,\\[4pt]
\mathrm{Cl}_{p,q+2}&\simeq&   \mathrm{Cl}_{q,p}\otimes_{\bbR}\mathrm{Cl}_{0,2} ,\\[4pt]
\label{LastProd}
\mathrm{Cl}_{p+1,q+1} &\simeq&  \mathrm{Cl}_{p,q}\otimes_{\bbR}\mathrm{Cl}_{1,1} .
\end{eqnarray*}
This in particular implies
$$
 \mathrm{Cl}_{p+8,q} 
 \simeq  
 \mathrm{Cl}_{p+4,q+4}\simeq\mathrm{Cl}_{p,q+8}
 \simeq
 \mathrm{Cl}_{p,q}\otimes_{\bbR}\mathrm{Mat}_{16}(\R),
$$
known as the Bott periodicity.

\subsection{How to use the generating function}
\label{How to use the generating function}

In order to illustrate our method and the role of generating functions,
let us give two simple proofs of the classical isomorphisms
$\mathrm{Cl}_{p+2,q} \simeq \mathrm{Cl}_{2,0}\otimes_{\bbR}\mathrm{Cl}_{q,p}$ and $\mathrm{Cl}_{p,q+2} \simeq \mathrm{Cl}_{q,p}\otimes_{\bbR}\mathrm{Cl}_{0,2}$.

The algebras $\mathrm{Cl}_{2,0}\otimes_{\bbR}\mathrm{Cl}_{q,p} $ and $\mathrm{Cl}_{q,p}\otimes_{\bbR}\mathrm{Cl}_{0,2} $
have respectively the following generating functions:
$$
\alpha(x)=x_1 x_2 + \sum_{3 \leq i\leq j \leq n +2} x_i x_j + \sum_{p+3\leq  i \leq n+2} x_i,
$$
and  
$$
\alpha'(x)= \sum_{1 \leq i\leq j \leq n } x_i x_j  + \sum_{p+1 \leq i\leq n} x_i +x_{n+1} x_{n+2} + x_{n+1} + x_{n+2}.
$$
It is easy to check that
the coordinate transformations
\begin{align}
\label{CliffordPeriodicityTransfo}
x'_1& = x _1 + x_3 + \ldots + x_{n+2},&& &  x_i'&=x_i, \; i \leq n, \nonumber \\
x'_2& = x _2 + x_{3} + \ldots + x_{n+2}, & &\mbox{and}&  x'_{n+1}& = x _1 + \ldots + x_{n+1}, \\
x_i'&=x_i, \; i>2,  &&& x'_{n+2}& = x _1 + \ldots + x_{n}+  x_{n+2}.  \nonumber
\end{align} 
send respectively $\alpha$ and $\alpha ' $ to the generating quadratic form of $\mathrm{Cl}_{p+2,q}$ and of $\mathrm{Cl}_{p,q+2}$. The last periodicity statement 
for the Clifford algebras, i.e. $\mathrm{Cl}_{p+1,q+1} \simeq \mathrm{Cl}_{1,1}\otimes_{\bbR}\mathrm{Cl}_{q,p}$ can be proved in a similar way.

\section{Triangulated graphs}

In this section, we present a way to interpret a cubic form on $\Z_2^n$
in term of a triangulated graph\footnote{
I am grateful to V. Ovsienko who explained me this method.} and reformulate our main results. 
This will allow us to find the simplest equivalent normal forms for the
cubic forms~$\alpha_{p,q}$, for which the periodicity statements are very
transparent.

\subsection{The definition} \label{GRAPSECT}

Consider an arbitrary cubic form on $\Z_2^n$:
$$
\alpha(x)=
\sum_{1 \leq i<{}j<{}k \leq n }
A_{ijk}\,x_ix_jx_k+
\sum_{1 \leq i<{}j \leq n }
B_{ij}\,x_ix_j+
\sum_{1\leq{} i \leq{}n}
C_{i}\,x_i.
$$
Note that this is precisely the form~(\ref{AlForm})
by we separate the terms for which some of the indices coincide.
We will associate a triangulated graph to every such function.
The definition is as follows.

\begin{df}
Given a cubic form $\alpha$, the corresponding triangulated graph is as follows.

\begin{enumerate}
\item

The set vertices of the graph coincides with the set $\{x_1,x_2,\ldots,x_n\}$.
Write $\bullet$ if $C_{i}=1$ and $\circ$ if $C_{i}=0$. 

\item

Two distinct vertices, $i$ and $j$, are joined by an edge if $B_{ij}=1$.

\item

Join by a triangle
\begin{tikzpicture}[scale=0.6]
\fill [color=gray!40] (0,0) -- (20:1) -- (340:1) -- cycle;
\draw [ultra thin] (0,0) -- (20: 1) ;
\draw [ultra thin] (0,0) -- (340: 1) ;
\draw [ultra thin] (20:1) -- (340: 1) ;
\end{tikzpicture} 
those (distinct) vertices $i,j,k$ for which $A_{ijk}=1$. 

\end{enumerate}
\end{df}

Note that the defined triangulated graph completely characterizes the
cubic form. 

\begin{ex}
Let us give elementary examples in the $2$-dimensional case.  
   
   \begin{enumerate}
\item

The first interesting case is that of the classical algebra of quaternions $\bbH$.
The quadratic form and the corresponding graph are as follows.
\\
   \begin{minipage}{0.65\textwidth}
         \begin{eqnarray*}
    {\alpha}^{\mathrm{Cl}}_{0,2} (x_1, x_2) &=&  x_1x_2  + x_1 + x_2 ,
      \end{eqnarray*}
   \end{minipage}
   \begin{minipage}{0.25\textwidth}
\begin{tikzpicture}[scale=1.2]
\tikzstyle{sommetNO}=[shape=circle,fill=white,draw=black,minimum size=0.5pt, inner sep=1.5pt]
\draw [ultra thick] (0,0) -- (0: 1) ;
\draw (0,0) node {$\bullet$};
\draw (0: 1) node {$\bullet$};
\draw (180: 1.5) node {$\longleftrightarrow$};
\draw (0:0) node[left] {\small{$x_1$}};
\draw (0:1) node[right] {\small{$x_2$}};
\end{tikzpicture} 
   \end{minipage}
 
\item
      The other interesting case is that the Clifford algebra $\mathrm{Cl}_{2,0}$.
The quadratic form and the corresponding graph are as follows.
\\
   \begin{minipage}{0.65\textwidth}
         \begin{eqnarray*}
    {\alpha}^{\mathrm{Cl}}_{2,0} (x_1, x_2) &=&  x_1x_2 ,
      \end{eqnarray*}
   \end{minipage}
   \begin{minipage}{0.25\textwidth}
\begin{tikzpicture}[scale=1.2]
\tikzstyle{sommetNO}=[shape=circle,fill=white,draw=black,minimum size=0.5pt, inner sep=1.5pt]
\node[draw=white, sommetNO] (X1) at (0,0) {};
\node[draw=white, sommetNO] (X2) at (0:1) {};
\draw [ultra thick] (X1) -- (X2) ;
\draw (180: 1.5) node {$\longleftrightarrow$};
\draw (X1) node[left] {\small{$x_1$}};
\draw (X2) node[right] {\small{$x_2$}};
\end{tikzpicture} 
   \end{minipage}
   
      \end{enumerate}
   
 \end{ex}  

\begin{ex}
Let us give several examples in the $3$-dimensional case.

\begin{enumerate}
\item
The first interesting case is that of the classical algebra of octonions $\bbO$.
The cubic form and the corresponding graph are as follows.
\\
   \begin{minipage}{0.65\textwidth}
         \begin{eqnarray*}
    {\alpha}_{0,3} (x_1, x_2, x_3) &=&  x_1x_2x_3 + x_1x_2 + x_1 x_3 + x_2 x_3 
    \\&& + x_1 + x_2 + x_3,
      \end{eqnarray*}
   \end{minipage}
   \begin{minipage}{0.25\textwidth}
\begin{tikzpicture}[scale=1.2]
\fill [color=gray!40] (0,0) -- (20:1) -- (340:1) -- cycle;
\draw [ultra thick] (0,0) -- (20: 1) ;
\draw [ultra thick] (0,0) -- (340: 1) ;
\draw [ultra thick] (20:1) -- (340: 1) ;
\draw (0,0) node {$\bullet$};
\draw (20: 1) node {$\bullet$};
\draw (340: 1) node {$\bullet$};
\draw (180: 1.5) node {$\longleftrightarrow$};
\draw (0:0) node[left] {\small{$x_1$}};
\draw (20:1) node[right] {\small{$x_2$}};
\draw (340:1) node[right] {\small{$x_3$}};
\end{tikzpicture} 
   \end{minipage}
   
   \medskip
   
\noindent
Amazingly, the above triangle contains the full information about the cubic form 
$ {\alpha}_{0,3}$ and therefore about the algebra $\bbO$.

\item

The algebra of split octonions has the following cubic form:
\\
   \begin{minipage}{0.65\textwidth}
         \begin{eqnarray*}
    {\alpha}_{1,2} (x_1, x_2, x_3) &=&  x_1x_2x_3 + x_1x_2 + x_1 x_3 + x_2 x_3 
    \\&&  + x_2 + x_3,
   \end{eqnarray*}
   \end{minipage}
   \begin{minipage}{0.25\textwidth}
\begin{tikzpicture}[scale=1.2]
\tikzstyle{sommetNO}=[shape=circle,fill=white,draw=black,minimum size=0.5pt, inner sep=1.5pt]
\fill [color=gray!40] (0,0) -- (20:1) -- (340:1) -- cycle;
\node[draw=white, sommetNO] (X1) at (0,0) {};
\draw [ultra thick] (X1) -- (20: 1) ;
\draw [ultra thick] (X1) -- (340: 1) ;
\draw [ultra thick] (20:1) -- (340: 1) ;
\draw (20: 1) node {$\bullet$};
\draw (340: 1) node {$\bullet$};
\draw (180: 1.5) node {$\longleftrightarrow$};
\draw (0:0) node[left] {\small{$x_1$}};
\draw (20:1) node[right] {\small{$x_2$}};
\draw (340:1) node[right] {\small{$x_3$}};
\end{tikzpicture} 
   \end{minipage}
   
   \medskip
 
 \item
 
 The ``trivial example'':
 \\
   \begin{minipage}{0.65\textwidth}
         \begin{eqnarray*}
    {\alpha} (x_1, x_2, x_3) &\equiv&  0 
      \end{eqnarray*}
   \end{minipage}
   \begin{minipage}{0.25\textwidth}
\begin{tikzpicture}[scale=1.2]
\tikzstyle{sommetNO}=[shape=circle,fill=white,draw=black,minimum size=0.5pt, inner sep=1.5pt]

\node[draw=white, sommetNO] (X1) at (0,0) {};
\node[draw=white, sommetNO] (X2) at (20:1) {};
\node[draw=white, sommetNO] (X3) at (340:1) {};

\draw (180: 1.5) node {$\longleftrightarrow$};
\draw (0:0) node[left] {\small{$x_1$}};
\draw (20:1) node[right] {\small{$x_2$}};
\draw (340:1) node[right] {\small{$x_3$}};
\end{tikzpicture} 
   \end{minipage}

\end{enumerate}
\end{ex}

\subsection{The forms $\tilde{\alpha}_{0,n}$ and $\tilde{\alpha}_{n,0}$}

Let us now introduce a series of cubic forms $ \tilde{\alpha}_{p,q}$.
We will prove in Section~\ref{subsectionisomorphisms} 
that they are equivalent to the forms ${\alpha}_{p,q}$.
The advantage of this new way to represent the cubic forms ${\alpha}_{p,q}$
consists in the fact that the corresponding graphs are very simple.
The periodicity properties of the algebras $\bbO_{n}$ and  $\bbO_{p,q}$
can be seen directly from the graphs.

Let us start with the case of signature $(0,n)$.

   \begin{df}
   The cubic forms $\tilde{\alpha}_{0,n}$ are defined as follows.
   
   \begin{enumerate}
   
   \item
   $\tilde{\alpha}_{0,3}={\alpha}_{0,3}$.
   \item
   The next cases are:
\\
   \begin{minipage}{0.65\textwidth}
       \begin{eqnarray*}
       \tilde{\alpha}_{0,4}(x_1,x_2,x_3,x_4)& =&  x_1x_3x_4 + x_1x_3 + x_1 x_4 + x_3 x_4 \\
     &&   + x_1 + x_3 + x_4,
	\end{eqnarray*}
   \end{minipage}
   \begin{minipage}{0.20\textwidth}
\begin{tikzpicture}[scale=1.2]
\tikzstyle{sommetNO}=[shape=circle,fill=white,draw=black,minimum size=0.5pt, inner sep=1.5pt]
\fill [color=gray!40] (0,0) -- (20:1) -- (340:1) -- cycle;
\draw [ultra thick] (0,0) -- (20: 1) ;
\draw [ultra thick] (0,0) -- (340: 1) ;
\draw [ultra thick] (20:1) -- (340: 1) ;
\draw (0,0) node {$\bullet$};
\draw (20: 1) node {$\bullet$};
\draw (340: 1) node {$\bullet$};
\draw (160: 1) node[sommetNO]{};
\draw (180: 2) node {$\longleftrightarrow$};
\draw (0:0) node[left] {\small{$x_1$}};
\draw (20:1) node[right] {\small{$x_4$}};
\draw (340:1) node[right] {\small{$x_3$}};
\draw (160: 1) node[below] {\small{$x_2$}};
\end{tikzpicture} 
   \end{minipage}
\\
   \begin{minipage}{0.65\textwidth}
   \begin{eqnarray*}
\tilde{\alpha}_{0,5}(x_1,\ldots ,x_5)&=&
x_1 x_2 x_3 + x_1 x_4 x_5 +  x_2 x_3 + x_1 x_4 \\
&& + x_1 x_5 + x_4 x_5 + x_1 + x_4 + x_5 ,
\end{eqnarray*}
   \end{minipage}
   \begin{minipage}{0.25\textwidth}
\begin{tikzpicture}[scale=1.2]
\tikzstyle{sommetNO}=[shape=circle,fill=white,draw=black,minimum size=0.5pt, inner sep=1.5pt]
\tikzstyle{sommetYES}=[shape=circle,fill=black,draw=black,minimum size=0.5pt, inner sep=1.5pt]
\fill [color=gray!40] (0,0) -- (20:1) -- (340:1) -- cycle;
\fill [color=gray!40] (0,0) -- (160:1) -- (200:1) -- cycle;
\tikzstyle{everypath}=[color=black, line width= ultra thick]
\node[draw=white, sommetNO] (X2) at (200:1) {};
\node[draw=white, sommetNO] (X3) at (160:1) {};
\draw[ultra thin] (0:0) node [sommetYES]{} -- (200:1) node[sommetNO]{};
\draw[ultra thin] (0:0) node [sommetYES]{} -- (160:1) node[sommetNO]{};
\draw [ultra thick] (X2)  -- (X3);
\draw [ultra thick] (0,0) -- (20: 1) ;
\draw [ultra thick] (0,0) -- (340: 1) ;
\draw [ultra thick] (20:1) -- (340: 1) ;
\draw (20: 1) node {$\bullet$};
\draw (340: 1) node {$\bullet$};

\draw (180: 2) node {$\longleftrightarrow$};
\draw (20:1) node[right] {\small{$x_5$}};
\draw (340: 1) node[right] {\small{$x_4$}};
\draw (200:1) node[left] {\small{$x_2$}};
\draw (160: 1) node[left] {\small{$x_3$}};
\draw (0: 0) node[below] {\small{$x_1$}};
\end{tikzpicture} 
   \end{minipage}
\\
   \begin{minipage}{0.65\textwidth}
      \begin{eqnarray*}
\tilde{\alpha}_{0,6}(x_1,\ldots ,x_6)& =& 
x_1 x_3 x_4 + x_1 x_5 x_6 + x_1 x_2+  x_3 x_4 + x_1 x_5 
\\ &&+ x_1 x_6 + x_5 x_6 + x_1 + x_2 +  x_5 + x_6 .
\end{eqnarray*}
   \end{minipage}
      \begin{minipage}{0.25\textwidth}
\begin{tikzpicture}[scale=1.2]
\tikzstyle{sommetNO}=[shape=circle,fill=white,draw=black,minimum size=0.5pt, inner sep=1.5pt]
\fill [color=gray!40] (0,0) -- (20:1) -- (340:1) -- cycle;
\fill [color=gray!40] (0,0) -- (160:1) -- (200:1) -- cycle;
\node[draw=white, sommetNO] (X3) at (200:1) {};
\node[draw=white, sommetNO] (X4) at (160:1) {};
\draw [ultra thick] (0,0) -- (20: 1) ;
\draw [ultra thick] (0,0) -- (340: 1) ;
\draw [ultra thick] (20:1) -- (340: 1) ;
\draw [ultra thin](0,0) -- (X3) ;
\draw  [ultra thin] (0,0) -- (X4) ;
\draw  [ultra thick] (X4)-- (X3) ;
\draw (0,0) node {$\bullet$};
\draw (20: 1) node {$\bullet$};
\draw (340: 1) node {$\bullet$};
\draw  [ultra thick] (0:0) -- (90: 0.5) ;
\draw (90: 0.5) node {$\bullet$};
\draw (180: 2) node {$\longleftrightarrow$};
\draw (20:1) node[right] {\small{$x_6$}};
\draw (340: 1) node[right] {\small{$x_5$}};
\draw (200:1) node[left] {\small{$x_3$}};
\draw (160: 1) node[left] {\small{$x_4$}};
\draw (0: 0) node[below] {\small{$x_1$}};
\draw (90: 0.5) node[above] {\small{$x_2$}};
\end{tikzpicture} 
   \end{minipage}
 
 \item
 In general, we have the following.  
\begin{eqnarray*}
 \tilde{\alpha}_{0,4k+3} (x_1, \ldots , x_{4+3k})  &=& 
   \tilde{\alpha}_{0,3} (x_1,x_2,  x_3) + 
   \sum_{i=1}^{k}     \left( x_1+\tilde{\alpha}_{0,5} (x_1, x_{4i},  \ldots, x_{4i+3}) \right) ,\\
 \tilde{\alpha}_{0,4k} (x_1, \ldots , x_{4k})  &=& 
   \tilde{\alpha}_{0,4} (x_1, \ldots , x_{4}) + 
   \sum_{i=1}^{k-1}  \left(   x_1 + \tilde{\alpha}_{0,5} (x_1, x_{4i+1} , \ldots, x_{4i+4}) \right), \\
 \tilde{\alpha}_{0,4k+1} (x_1, \ldots , x_{4k+1})  &=&  \tilde{\alpha}_{0,5} (x_1, \ldots , x_{5})  + 
\sum_{i=1}^{k-1}   \left(    x_1+ \tilde{\alpha}_{0,5} (x_1, x_{4i+2} , \ldots , x_{4i+5}) \right),   \\
 \tilde{\alpha}_{0,4k+2} (x_1, \ldots , x_{4k+2})  &=& 
   \tilde{\alpha}_{0,6} (x_1,  \ldots , x_6) +
    \sum_{i=1}^{k-1}  \left( x_1+ \tilde{\alpha}_{0,5} (x_1, x_{4i+3} ,  \ldots , x_{4i+6})  \right). 
\end{eqnarray*}

 \end{enumerate}
\end{df}

      The table below gives a series of examples of
      defined cubic forms. 
        
        \medskip
 
\begin{tabular}{ccc}
\begin{tikzpicture}[scale=1.2]
\fill [color=gray!40] (0,0) -- (20:1) -- (340:1) -- cycle;
\draw [ultra thick] (0,0) -- (20: 1) ;
\draw [ultra thick] (0,0) -- (340: 1) ;
\draw [ultra thick] (20:1) -- (340: 1) ;
\draw (0,0) node {$\bullet$};
\draw (20: 1) node {$\bullet$};
\draw (340: 1) node {$\bullet$};
\draw (180: 0.8) node {$\tilde{\alpha}_{0,3}$};
\end{tikzpicture} 
&
\begin{tikzpicture}[scale=1.2]
\tikzstyle{sommetNO}=[shape=circle,fill=white,draw=black,minimum size=0.5pt, inner sep=1.5pt]
\fill [color=gray!40] (0,0) -- (10:1) -- (50:1) -- cycle;
\fill [color=gray!40] (0,0) -- (170:1) -- (130:1) -- cycle;
\fill [color=gray!40] (0,0) -- (250:1) -- (290:1) -- cycle;
\node[draw=white, sommetNO] (X3) at (250:1) {};
\node[draw=white, sommetNO] (X2) at (290:1) {};
\draw [ultra thick] (0,0) -- (10: 1) ;
\draw [ultra thick] (0,0) -- (50: 1) ;
\draw [ultra thick] (10:1) -- (50: 1) ;
\draw [ultra thick] (0,0) -- (130: 1) ;
\draw [ultra thick] (0,0) -- (170: 1) ;
\draw [ultra thick] (130:1) -- (170: 1) ;
\draw [ultra thin](0,0) -- (X3) ;
\draw  [ultra thin] (0,0) -- (X2) ;
\draw  [ultra thick] (X2) -- (X3) ;
\draw (0,0) node {$\bullet$};
\draw (10: 1) node {$\bullet$};
\draw (50: 1) node {$\bullet$};
\draw (130: 1) node {$\bullet$};
\draw (170: 1) node {$\bullet$};
\draw (180: 1.8) node {$\tilde{\alpha}_{0,7}$};
\end{tikzpicture} 
&
\begin{tikzpicture}[scale=1.2]
\tikzstyle{sommetNO}=[shape=circle,fill=white,draw=black,minimum size=0.5pt, inner sep=1.5pt]
\fill [color=gray!40] (0,0) -- (34:1) -- (74:1) -- cycle;
\fill [color=gray!40] (0,0) -- (106:1) -- (146:1) -- cycle;
\fill [color=gray!40] (0,0) -- (322:1) -- (2:1) -- cycle;
\fill [color=gray!40] (0,0) -- (178:1) -- (218:1) -- cycle;
\fill [color=gray!40] (0,0) -- (250:1) -- (290:1) -- cycle;
\node[draw=white, sommetNO] (X3) at (34:1) {};
\node[draw=white, sommetNO] (X2) at (74:1) {};
\node[draw=white, sommetNO] (X4) at (106:1) {};
\node[draw=white, sommetNO] (X5) at (146:1) {};
\draw (0: 0) node {$\bullet$};
\draw [ultra thick] (0,0) -- (2: 1) ;
\draw [ultra thick] (0,0) -- (322: 1) ;
\draw [ultra thick] (2:1) -- (322: 1) ;
\draw (322: 1) node {$\bullet$};
\draw (2: 1) node {$\bullet$};
\draw [ultra thin](0,0) -- (X2) ;
\draw  [ultra thin] (0,0) -- (X3) ;
\draw  [ultra thick] (X2) -- (X3) ;
\draw [ultra thin] (0,0) -- (X5) ;
\draw [ultra thin] (0,0) -- (X4) ;
\draw [ultra thick] (X4) -- (X5) ;
\draw [ultra thick](0,0) -- (178: 1) ;
\draw  [ultra thick] (0,0) -- (218: 1) ;
\draw  [ultra thick] (178:1) -- (218: 1) ;
\draw (218: 1) node {$\bullet$};
\draw (178: 1) node {$\bullet$};
\draw [ultra thick](0,0) -- (250: 1) ;
\draw  [ultra thick] (0,0) -- (290: 1) ;
\draw  [ultra thick] (250:1) -- (290: 1) ;
\draw (250: 1) node {$\bullet$};
\draw (290: 1) node {$\bullet$};
\draw (180: 1.8) node {$\tilde{\alpha}_{0,11}$};
\end{tikzpicture} 
 \rule[-15pt]{0pt}{40pt}
\\
\begin{tikzpicture}[scale=1.2]
\tikzstyle{sommetNO}=[shape=circle,fill=white,draw=black,minimum size=0.5pt, inner sep=1.5pt]
\fill [color=gray!40] (0,0) -- (20:1) -- (340:1) -- cycle;
\draw [ultra thick] (0,0) -- (20: 1) ;
\draw [ultra thick] (0,0) -- (340: 1) ;
\draw [ultra thick] (20:1) -- (340: 1) ;
\draw (0,0) node {$\bullet$};
\draw (20: 1) node {$\bullet$};
\draw (340: 1) node {$\bullet$};
\draw (90: 0.5) node [sommetNO]{};
\draw (180: 0.8) node {$\tilde{\alpha}_{0,4}$};
\end{tikzpicture} 
&
\begin{tikzpicture}[scale=1.2]
\tikzstyle{sommetNO}=[shape=circle,fill=white,draw=black,minimum size=0.5pt, inner sep=1.5pt]
\fill [color=gray!40] (0,0) -- (10:1) -- (50:1) -- cycle;
\fill [color=gray!40] (0,0) -- (170:1) -- (130:1) -- cycle;
\fill [color=gray!40] (0,0) -- (250:1) -- (290:1) -- cycle;
\node[draw=white, sommetNO] (X3) at (250:1) {};
\node[draw=white, sommetNO] (X2) at (290:1) {};
\draw [ultra thick] (0,0) -- (10: 1) ;
\draw [ultra thick] (0,0) -- (50: 1) ;
\draw [ultra thick] (10:1) -- (50: 1) ;
\draw [ultra thick] (0,0) -- (130: 1) ;
\draw [ultra thick] (0,0) -- (170: 1) ;
\draw [ultra thick] (130:1) -- (170: 1) ;
\draw [ultra thin](0,0) -- (X2) ;
\draw  [ultra thin] (0,0) -- (X3) ;
\draw  [ultra thick] (X2) -- (X3) ;
\draw (0,0) node {$\bullet$};
\draw (10: 1) node {$\bullet$};
\draw (50: 1) node {$\bullet$};
\draw (130: 1) node {$\bullet$};
\draw (170: 1) node {$\bullet$};
\draw (90: 0.7) node [sommetNO]{};
\draw (180: 1.8) node {$\tilde{\alpha}_{0,8}$};
\end{tikzpicture} 
&
\begin{tikzpicture}[scale=1.2]
\tikzstyle{sommetNO}=[shape=circle,fill=white,draw=black,minimum size=0.5pt, inner sep=1.5pt]
\fill [color=gray!40] (0,0) -- (34:1) -- (74:1) -- cycle;
\fill [color=gray!40] (0,0) -- (106:1) -- (146:1) -- cycle;
\fill [color=gray!40] (0,0) -- (322:1) -- (2:1) -- cycle;
\fill [color=gray!40] (0,0) -- (178:1) -- (218:1) -- cycle;
\fill [color=gray!40] (0,0) -- (250:1) -- (290:1) -- cycle;

\node[draw=white, sommetNO] (X3) at (34:1) {};
\node[draw=white, sommetNO] (X2) at (74:1) {};
\node[draw=white, sommetNO] (X4) at (106:1) {};
\node[draw=white, sommetNO] (X5) at (146:1) {};
\draw (0: 0) node {$\bullet$};
\draw [ultra thick] (0,0) -- (2: 1) ;
\draw [ultra thick] (0,0) -- (322: 1) ;
\draw [ultra thick] (2:1) -- (322: 1) ;
\draw (322: 1) node {$\bullet$};
\draw (2: 1) node {$\bullet$};
\draw [ultra thin](0,0) -- (X2) ;
\draw  [ultra thin] (0,0) -- (X3) ;
\draw  [ultra thick] (X2) -- (X3) ;
\draw [ultra thin] (0,0) -- (X5) ;
\draw [ultra thin] (0,0) -- (X4) ;
\draw [ultra thick] (X4) -- (X5) ;
\draw [ultra thick](0,0) -- (178: 1) ;
\draw  [ultra thick] (0,0) -- (218: 1) ;
\draw  [ultra thick] (178:1) -- (218: 1) ;
\draw (218: 1) node {$\bullet$};
\draw (178: 1) node {$\bullet$};
\draw [ultra thick](0,0) -- (250: 1) ;
\draw  [ultra thick] (0,0) -- (290: 1) ;
\draw  [ultra thick] (250:1) -- (290: 1) ;
\draw (250: 1) node {$\bullet$};
\draw (290: 1) node {$\bullet$};
\draw (90: 0.75) node [sommetNO]{};
\draw (180: 1.8) node {$\tilde{\alpha}_{0,12}$};
\end{tikzpicture} 
 \rule[-15pt]{0pt}{40pt}
\\
\begin{tikzpicture}[scale=1.2]
\tikzstyle{sommetNO}=[shape=circle,fill=white,draw=black,minimum size=0.5pt, inner sep=1.5pt]
\fill [color=gray!40] (0,0) -- (20:1) -- (340:1) -- cycle;
\fill [color=gray!40] (0,0) -- (160:1) -- (200:1) -- cycle;
\node[draw=white, sommetNO] (X3) at (160:1) {};
\node[draw=white, sommetNO] (X2) at (200:1) {};
\draw [ultra thick] (0,0) -- (20: 1) ;
\draw [ultra thick] (0,0) -- (340: 1) ;
\draw [ultra thick] (20:1) -- (340: 1) ;
\draw [ultra thin](0,0) -- (X2) ;
\draw  [ultra thin] (0,0) -- (X3) ;
\draw  [ultra thick] (X2) -- (X3) ;
\draw (0,0) node {$\bullet$};
\draw (20: 1) node {$\bullet$};
\draw (340: 1) node {$\bullet$};
\draw (180: 1.8) node {$\tilde{\alpha}_{0,5}$};
\end{tikzpicture} 
&
\begin{tikzpicture}[scale=1.2]
\tikzstyle{sommetNO}=[shape=circle,fill=white,draw=black,minimum size=0.5pt, inner sep=1.5pt]
\fill [color=gray!40] (0,0) -- (20:1) -- (340:1) -- cycle;
\fill [color=gray!40] (0,0) -- (160:1) -- (200:1) -- cycle;
\fill [color=gray!40] (0,0) -- (250:1) -- (290:1) -- cycle;
\fill [color=gray!40] (0,0) -- (110:1) -- (70:1) -- cycle;
\node[draw=white, sommetNO] (X3) at (70:1) {};
\node[draw=white, sommetNO] (X2) at (110:1) {};
\node[draw=white, sommetNO] (X4) at (160:1) {};
\node[draw=white, sommetNO] (X5) at (200:1) {};
\draw [ultra thin](0,0) -- (X2) ;
\draw  [ultra thin] (0,0) -- (X3) ;
\draw  [ultra thick] (X2) -- (X3) ;
\draw [ultra thin](0,0) -- (X4) ;
\draw  [ultra thin] (0,0) -- (X5) ;
\draw  [ultra thick] (X4) -- (X5) ;
\draw [ultra thick] (0,0) -- (20: 1) ;
\draw [ultra thick] (0,0) -- (340: 1) ;
\draw [ultra thick] (20:1) -- (340: 1) ;
\draw (0,0) node {$\bullet$};
\draw (20: 1) node {$\bullet$};
\draw (340: 1) node {$\bullet$};
\draw [ultra thick](0,0) -- (250: 1) ;
\draw  [ultra thick] (0,0) -- (290: 1) ;
\draw  [ultra thick] (250:1) -- (290: 1) ;
\draw (250: 1) node {$\bullet$};
\draw (290: 1) node {$\bullet$};
\draw (180: 1.8) node {$\tilde{\alpha}_{0,9}$};
\end{tikzpicture} 
&
\begin{tikzpicture}[scale=1.2]
\tikzstyle{sommetNO}=[shape=circle,fill=white,draw=black,minimum size=0.5pt, inner sep=1.5pt]
\fill [color=gray!40] (0,0) -- (40:1) -- (80:1) -- cycle;
\fill [color=gray!40] (0,0) -- (140:1) -- (100:1) -- cycle;
\fill [color=gray!40] (0,0) -- (160:1) -- (200:1) -- cycle;
\fill [color=gray!40] (0,0) -- (20:1) -- (340:1) -- cycle;
\fill [color=gray!40] (0,0) -- (220:1) -- (260:1) -- cycle;
\fill [color=gray!40] (0,0) -- (320:1) -- (280:1) -- cycle;
\node[draw=white, sommetNO] (X3) at (40:1) {};
\node[draw=white, sommetNO] (X2) at (80:1) {};
\node[draw=white, sommetNO] (X4) at (100:1) {};
\node[draw=white, sommetNO] (X5) at (140:1) {};
\node[draw=white, sommetNO] (X6) at (160:1) {};
\node[draw=white, sommetNO] (X7) at (200:1) {};
\draw [ultra thin](0,0) -- (X2) ;
\draw  [ultra thin] (0,0) -- (X3) ;
\draw  [ultra thick] (X2) -- (X3) ;
\draw [ultra thin](0,0) -- (X4) ;
\draw  [ultra thin] (0,0) -- (X5) ;
\draw  [ultra thick] (X4) -- (X5) ;
\draw [ultra thin](0,0) -- (X6) ;
\draw  [ultra thin] (0,0) -- (X7) ;
\draw  [ultra thick] (X6) -- (X7) ;
\draw (0,0) node {$\bullet$};
\draw [ultra thick] (0,0) -- (20: 1) ;
\draw [ultra thick] (0,0) -- (340: 1) ;
\draw [ultra thick] (20:1) -- (340: 1) ;
\draw (20: 1) node {$\bullet$};
\draw (340: 1) node {$\bullet$};
\draw [ultra thick](0,0) -- (220: 1) ;
\draw  [ultra thick] (0,0) -- (260: 1) ;
\draw  [ultra thick] (220:1) -- (260: 1) ;
\draw (220: 1) node {$\bullet$};
\draw (260: 1) node {$\bullet$};
\draw [ultra thick](0,0) -- (280: 1) ;
\draw  [ultra thick] (0,0) -- (320: 1) ;
\draw  [ultra thick] (320:1) -- (280: 1) ;
\draw (280: 1) node {$\bullet$};
\draw (320: 1) node {$\bullet$};
\draw (180: 1.8) node {$\tilde{\alpha}_{0,13}$};
\end{tikzpicture} 
 \rule[-15pt]{0pt}{40pt}
\end{tabular}

\begin{tabular}{ccc}
\begin{tikzpicture}[scale=1.2]
\tikzstyle{sommetNO}=[shape=circle,fill=white,draw=black,minimum size=0.5pt, inner sep=1.5pt]
\fill [color=gray!40] (0,0) -- (20:1) -- (340:1) -- cycle;
\fill [color=gray!40] (0,0) -- (160:1) -- (200:1) -- cycle;
\node[draw=white, sommetNO] (X3) at (160:1) {};
\node[draw=white, sommetNO] (X2) at (200:1) {};
\draw [ultra thick] (0,0) -- (20: 1) ;
\draw [ultra thick] (0,0) -- (340: 1) ;
\draw [ultra thick] (20:1) -- (340: 1) ;
\draw [ultra thin](0,0) -- (X2) ;
\draw  [ultra thin] (0,0) -- (X3) ;
\draw  [ultra thick] (X2) -- (X3) ;
\draw (0,0) node {$\bullet$};
\draw (20: 1) node {$\bullet$};
\draw (340: 1) node {$\bullet$};
\draw  [ultra thick] (0:0) -- (90: 0.5) ;
\draw (90: 0.5) node {$\bullet$};
\draw (180: 1.8) node {$\tilde{\alpha}_{0,6}$};
\end{tikzpicture} 
&
\begin{tikzpicture}[scale=1.2]

\tikzstyle{sommetNO}=[shape=circle,fill=white,draw=black,minimum size=0.5pt, inner sep=1.5pt]
\fill [color=gray!40] (0,0) -- (20:1) -- (340:1) -- cycle;
\fill [color=gray!40] (0,0) -- (160:1) -- (200:1) -- cycle;
\fill [color=gray!40] (0,0) -- (250:1) -- (290:1) -- cycle;
\fill [color=gray!40] (0,0) -- (110:1) -- (70:1) -- cycle;
\node[draw=white, sommetNO] (X3) at (70:1) {};
\node[draw=white, sommetNO] (X2) at (110:1) {};

\node[draw=white, sommetNO] (X4) at (160:1) {};
\node[draw=white, sommetNO] (X5) at (200:1) {};
\draw [ultra thin](0,0) -- (X2) ;
\draw  [ultra thin] (0,0) -- (X3) ;
\draw  [ultra thick] (X2) -- (X3) ;
\draw [ultra thin](0,0) -- (X4) ;
\draw  [ultra thin] (0,0) -- (X5) ;
\draw  [ultra thick] (X4) -- (X5) ; 
\draw [ultra thick] (0,0) -- (20: 1) ;
\draw [ultra thick] (0,0) -- (340: 1) ;
\draw [ultra thick] (20:1) -- (340: 1) ;
\draw (0,0) node {$\bullet$};
\draw (20: 1) node {$\bullet$};
\draw (340: 1) node {$\bullet$};
\draw [ultra thick](0,0) -- (250: 1) ;
\draw  [ultra thick] (0,0) -- (290: 1) ;
\draw  [ultra thick] (250:1) -- (290: 1) ;
\draw (250: 1) node {$\bullet$};
\draw (290: 1) node {$\bullet$};
\draw (135: 1) node {$\bullet$};
\draw  [ultra thick] (0:0) -- (135: 1) ;
\draw (180: 1.8) node {$\tilde{\alpha}_{0,10}$};
\end{tikzpicture} 
&
\begin{tikzpicture}[scale=1.2]
\tikzstyle{sommetNO}=[shape=circle,fill=white,draw=black,minimum size=0.5pt, inner sep=1.5pt]
\fill [color=gray!40] (0,0) -- (40:1) -- (80:1) -- cycle;
\fill [color=gray!40] (0,0) -- (140:1) -- (100:1) -- cycle;
\fill [color=gray!40] (0,0) -- (160:1) -- (200:1) -- cycle;
\fill [color=gray!40] (0,0) -- (20:1) -- (340:1) -- cycle;
\fill [color=gray!40] (0,0) -- (220:1) -- (260:1) -- cycle;
\fill [color=gray!40] (0,0) -- (320:1) -- (280:1) -- cycle;

\node[draw=white, sommetNO] (X3) at (40:1) {};
\node[draw=white, sommetNO] (X2) at (80:1) {};
\node[draw=white, sommetNO] (X4) at (100:1) {};
\node[draw=white, sommetNO] (X5) at (140:1) {};
\node[draw=white, sommetNO] (X6) at (160:1) {};
\node[draw=white, sommetNO] (X7) at (200:1) {};

\draw [ultra thin](0,0) -- (X2) ;
\draw  [ultra thin] (0,0) -- (X3) ;
\draw  [ultra thick] (X2) -- (X3) ;
\draw [ultra thin](0,0) -- (X4) ;
\draw  [ultra thin] (0,0) -- (X5) ;
\draw  [ultra thick] (X4) -- (X5) ;
\draw [ultra thin](0,0) -- (X6) ;
\draw  [ultra thin] (0,0) -- (X7) ;
\draw  [ultra thick] (X6) -- (X7) ;

\draw (0,0) node {$\bullet$};
\draw [ultra thick] (0,0) -- (20: 1) ;
\draw [ultra thick] (0,0) -- (340: 1) ;
\draw [ultra thick] (20:1) -- (340: 1) ;
\draw (20: 1) node {$\bullet$};
\draw (340: 1) node {$\bullet$};
\draw [ultra thick](0,0) -- (220: 1) ;
\draw  [ultra thick] (0,0) -- (260: 1) ;
\draw  [ultra thick] (220:1) -- (260: 1) ;
\draw (220: 1) node {$\bullet$};
\draw (260: 1) node {$\bullet$};
\draw [ultra thick](0,0) -- (280: 1) ;
\draw  [ultra thick] (0,0) -- (320: 1) ;
\draw  [ultra thick] (320:1) -- (280: 1) ;
\draw (280: 1) node {$\bullet$};
\draw (320: 1) node {$\bullet$};
\draw  [ultra thick] (0:0) -- (90: 1) ;
\draw (90: 1) node {$\bullet$};
\draw (180: 1.8) node {$\tilde{\alpha}_{0,14}$};
\end{tikzpicture} 
  \rule[-15pt]{0pt}{40pt}
\end{tabular}
\begin{center}
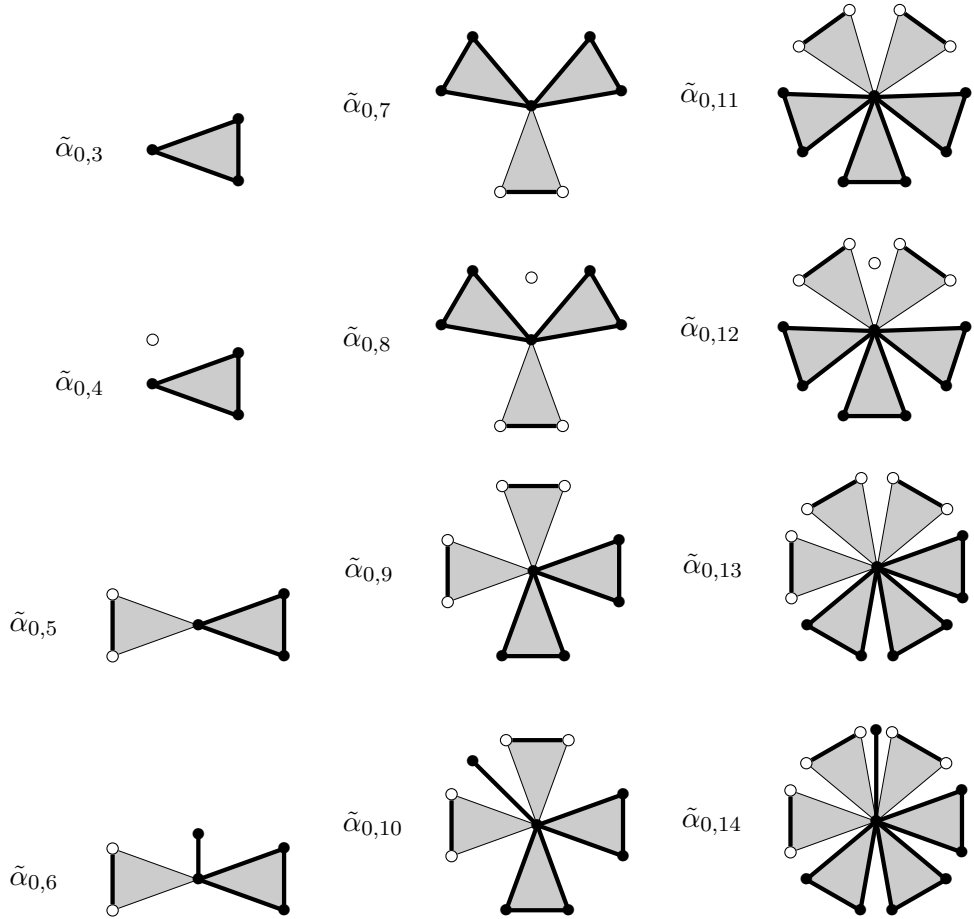
Figure 2. Examples of the cubic form $\tilde{\alpha}_{0,n} $ for $n\in \{ 3, \ldots, 14\} $. 
\end{center}

The property of periodicity modulo $4$  is quite obvious.
\begin{df}
\label{definitiona_{n,0}}
The forms $\tilde{\alpha}_{n,0}$ are defined according the following simple rule:
$$
\tilde{\alpha}_{n,0}(x_1,\ldots,x_n):=\tilde{\alpha}_{0,n}(x_1,\ldots,x_n)+x_1.
$$
\end{df}

\subsection{The forms $ \tilde{\alpha}_{p,q}$}

The cubic forms $\tilde{\alpha}_{p,q}$ with signature $(p,q)$
such that $p>0$ and $q>0$ are defined as follows.

\begin{df}
\begin{enumerate}

\item
   
The first eleven cases are defined as follows.
  
 \bigskip

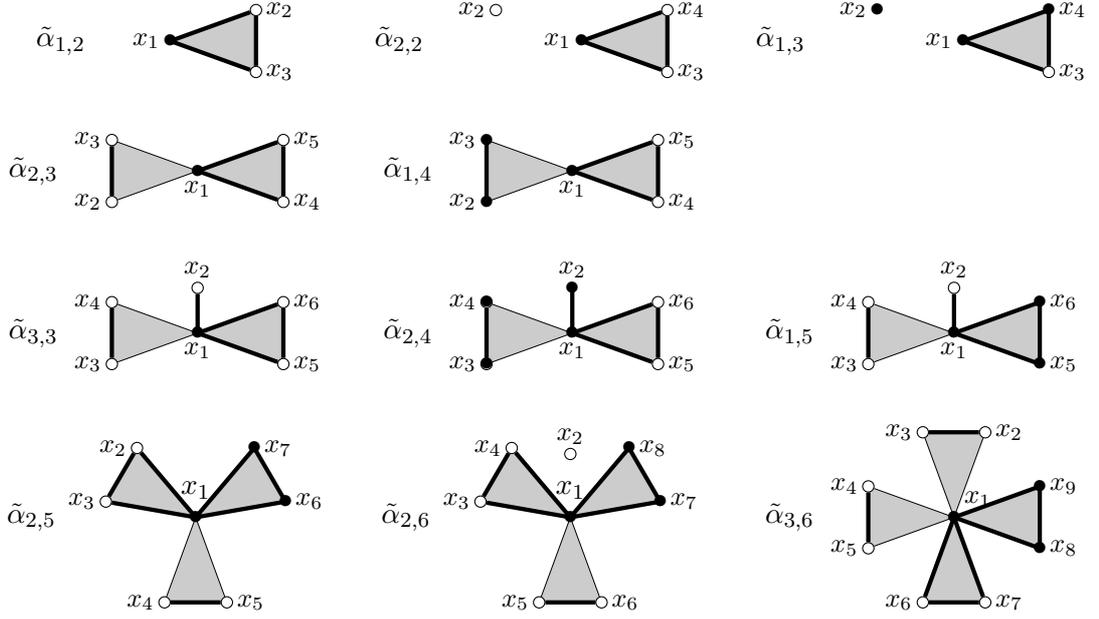
\begin{tabular}{ccc}
 \vspace{0.3cm}
\begin{tikzpicture}[scale=1.2]
\tikzstyle{sommetNO}=[shape=circle,fill=white,draw=black,minimum size=0.5pt, inner sep=1.5pt]
\fill [color=gray!40] (0,0) -- (20:1) -- (340:1) -- cycle;
\node[draw=white, sommetNO] (X3) at (340:1) {};
\node[draw=white, sommetNO] (X2) at (20:1) {};
\draw [ultra thick] (0,0) -- (X2) ;
\draw [ultra thick] (0,0) -- (X3) ;
\draw [ultra thick] (X2) -- (X3) ;
\draw (0,0) node {$\bullet$};
\draw (0:0) node[left] {\small{$x_1$}};
\draw (20:1) node[right] {\small{$x_2$}};
\draw (340:1) node[right] {\small{$x_3$}};
\draw (180: 1.2) node {$\tilde{\alpha}_{1,2} $};
\end{tikzpicture} 
 
 &
 
\begin{tikzpicture}[scale=1.2]
\tikzstyle{sommetNO}=[shape=circle,fill=white,draw=black,minimum size=0.5pt, inner sep=1.5pt]
\fill [color=gray!40] (0,0) -- (20:1) -- (340:1) -- cycle;
\node[draw=white, sommetNO] (X2) at (20:1) {};
\node[draw=white, sommetNO] (X3) at (340:1) {};
\node[draw=white, sommetNO] (X4) at (160:1) {};
\draw [ultra thick] (0,0) -- (X2) ;
\draw [ultra thick] (0,0) -- (X3) ;
\draw [ultra thick] (X2) -- (X3) ;
\draw (0,0) node {$\bullet$};
\draw (0,0) node[left] {\small{$x_1$}};
\draw (20:1) node[right] {\small{$x_4$}};
\draw (340:1) node[right] {\small{$x_3$}};
\draw (160:1 ) node[left] {\small{$x_2$}};
\draw (180: 2) node {$\tilde{\alpha}_{2,2} $};
\end{tikzpicture} 
&
\begin{tikzpicture}[scale=1.2]
\tikzstyle{sommetNO}=[shape=circle,fill=white,draw=black,minimum size=0.5pt, inner sep=1.5pt]
\fill [color=gray!40] (0,0) -- (20:1) -- (340:1) -- cycle;
\node[draw=white, sommetNO] (X3) at (340:1) {};
\draw [ultra thick] (0,0) -- (20: 1) ;
\draw [ultra thick] (0,0) -- (X3) ;
\draw [ultra thick] (20:1) -- (X3) ;
\draw (0,0) node {$\bullet$};
\draw (20: 1) node {$\bullet$};
\draw (160: 1) node {$\bullet$};
\draw (0:0) node[left] {\small{$x_1$}};
\draw (20:1) node[right] {\small{$x_4$}};
\draw (340:1) node[right] {\small{$x_3$}};
\draw (160:1 ) node[left] {\small{$x_2$}};
\draw (180: 2) node {$\tilde{\alpha}_{1,3} $};
\end{tikzpicture}

\\ \vspace{0.3cm}
\begin{tikzpicture}[scale=1.2]
\tikzstyle{sommetNO}=[shape=circle,fill=white,draw=black,minimum size=0.5pt, inner sep=1.5pt]
\fill [color=gray!40] (0,0) -- (20:1) -- (340:1) -- cycle;
\fill [color=gray!40] (0,0) -- (160:1) -- (200:1) -- cycle;
\node[draw=white, sommetNO] (X2) at (20:1) {};
\node[draw=white, sommetNO] (X3) at (340:1) {};
\node[draw=white, sommetNO] (X4) at (160:1) {};
\node[draw=white, sommetNO] (X5) at (200:1) {};
\draw [ultra thick] (0,0) -- (X2) ;
\draw [ultra thick] (0,0) -- (X3) ;
\draw [ultra thick] (X2) -- (X3) ;
\draw [ultra thin](0,0) -- (X4) ;
\draw  [ultra thin] (0,0) -- (X5) ;
\draw  [ultra thick] (X4) -- (X5) ;
\draw (0,0) node {$\bullet$};
\draw (20:1) node[right] {\small{$x_5$}};
\draw (340: 1) node[right] {\small{$x_4$}};
\draw (200:1) node[left] {\small{$x_2$}};
\draw (160: 1) node[left] {\small{$x_3$}};
\draw (0:0) node[below] {\small{$x_1$}};
\draw (180: 1.8) node {$\tilde{\alpha}_{2,3}$};
\end{tikzpicture} 
& 

\begin{tikzpicture}[scale=1.2]
\tikzstyle{sommetNO}=[shape=circle,fill=white,draw=black,minimum size=0.5pt, inner sep=1.5pt]
\fill [color=gray!40] (0,0) -- (20:1) -- (340:1) -- cycle;
\fill [color=gray!40] (0,0) -- (160:1) -- (200:1) -- cycle;
\node[draw=white, sommetNO] (X2) at (20:1) {};
\node[draw=white, sommetNO] (X3) at (340:1) {};
\draw [ultra thick] (0,0) -- (X2) ;
\draw [ultra thick] (0,0) -- (X3) ;
\draw [ultra thick] (X2) -- (X3) ;
\draw [ultra thin](0,0) -- (160: 1) ;
\draw  [ultra thin] (0,0) -- (200: 1) ;
\draw  [ultra thick] (160:1) -- (200: 1) ;
\draw (0,0) node {$\bullet$};
\draw (160: 1) node {$\bullet$};
\draw (200: 1) node {$\bullet$};
\draw (20:1) node[right] {\small{$x_5$}};
\draw (340: 1) node[right] {\small{$x_4$}};
\draw (200:1) node[left] {\small{$x_2$}};
\draw (160: 1) node[left] {\small{$x_3$}};
\draw (0:0) node[below] {\small{$x_1$}};
\draw (180: 1.8) node {$ \tilde{\alpha}_{1,4}$};
\end{tikzpicture}

&
\\ \vspace{0.3cm}
\begin{tikzpicture}[scale=1.2]
\tikzstyle{sommetNO}=[shape=circle,fill=white,draw=black,minimum size=0.5pt, inner sep=1.5pt]
\fill [color=gray!40] (0,0) -- (20:1) -- (340:1) -- cycle;
\fill [color=gray!40] (0,0) -- (160:1) -- (200:1) -- cycle;
\node[draw=white, sommetNO] (X2) at (90:0.5) {};
\node[draw=white, sommetNO] (X3) at (200:1) {};
\node[draw=white, sommetNO] (X4) at (160:1) {};
\node[draw=white, sommetNO] (X5) at (340:1) {};
\node[draw=white, sommetNO] (X6) at (20:1) {};
\draw [ultra thick] (0,0) -- (X6) ;
\draw [ultra thick] (0,0) -- (X5) ;
\draw [ultra thick] (X6) -- (X5) ;
\draw [ultra thin](0,0) -- (X4) ;
\draw  [ultra thin] (0,0) -- (X3) ;
\draw  [ultra thick] (X4) -- (X3) ;
\draw (0,0) node {$\bullet$};
\draw  [ultra thick] (0:0) -- (X2) ;
\draw (X6) node[right] {\small{$x_6$}};
\draw (X5) node[right] {\small{$x_5$}};
\draw (X3) node[left] {\small{$x_3$}};
\draw (X4) node[left] {\small{$x_4$}};
\draw (0,0) node[below] {\small{$x_1$}};
\draw (90: 0.5) node[above] {\small{$x_2$}};
\draw (180: 1.8) node {$\tilde{\alpha}_{3,3}$};
\end{tikzpicture} 
&
\begin{tikzpicture}[scale=1.2]
\tikzstyle{sommetNO}=[shape=circle,fill=white,draw=black,minimum size=0.5pt, inner sep=1.5pt]
\fill [color=gray!40] (0,0) -- (20:1) -- (340:1) -- cycle;
\fill [color=gray!40] (0,0) -- (160:1) -- (200:1) -- cycle;
\node[draw=white, sommetNO] (X2) at (20:1) {};
\node[draw=white, sommetNO] (X3) at (340:1) {};
\node[draw=white, sommetNO] (X4) at (160:1) {};
\node[draw=white, sommetNO] (X5) at (200:1) {};
\draw [ultra thick] (0,0) -- (X2) ;
\draw [ultra thick] (0,0) -- (X3) ;
\draw [ultra thick] (X2) -- (X3) ;
\draw [ultra thin](0,0) -- (X4) ;
\draw  [ultra thin] (0,0) -- (X5) ;
\draw  [ultra thick] (X4) -- (X5) ; 
\draw (0,0) node {$\bullet$};
\draw  [ultra thick] (0:0) -- (90: 0.5) ;
\draw (X4) node {$\bullet$};
\draw (X5) node {$\bullet$};
\draw (90: 0.5) node {$\bullet$};
\draw (20:1) node[right] {\small{$x_6$}};
\draw (340: 1) node[right] {\small{$x_5$}};
\draw (200:1) node[left] {\small{$x_3$}};
\draw (160: 1) node[left] {\small{$x_4$}};
\draw (0,0) node[below] {\small{$x_1$}};
\draw (90: 0.5) node[above] {\small{$x_2$}};
\draw (180: 1.8) node {$\tilde{\alpha}_{2,4}$};
\end{tikzpicture} 
&
\begin{tikzpicture}[scale=1.2]
\tikzstyle{sommetNO}=[shape=circle,fill=white,draw=black,minimum size=0.5pt, inner sep=1.5pt]
\fill [color=gray!40] (0,0) -- (20:1) -- (340:1) -- cycle;
\fill [color=gray!40] (0,0) -- (160:1) -- (200:1) -- cycle;
\node[draw=white, sommetNO] (X2) at (90:0.5) {};
\node[draw=white, sommetNO] (X3) at (160:1) {};
\node[draw=white, sommetNO] (X4) at (200:1) {};
\draw [ultra thick] (0,0) -- (X2) ;
\draw [ultra thick] (0,0) -- (340:1) ;
\draw [ultra thick] (20:1) -- (340:1) ;
\draw [ultra thin](0,0) -- (X3) ;
\draw [ultra thin](0,0) -- (X4) ;
\draw  [ultra thick] (0,0) -- (20:1) ;
\draw  [ultra thick] (X3) -- (X4) ;
\draw (0,0) node {$\bullet$};
\draw (20: 1) node {$\bullet$};
\draw (340: 1) node {$\bullet$};
\draw (20:1) node[right] {\small{$x_6$}};
\draw (340:1) node[right] {\small{$x_5$}};
\draw (200:1) node[left] {\small{$x_3$}};
\draw (160: 1) node[left] {\small{$x_4$}};
\draw (0,0) node[below] {\small{$x_1$}};
\draw (X2) node[above] {\small{$x_2$}};
\draw (180: 1.8) node {$\tilde{\alpha}_{1,5}$};
\end{tikzpicture} 
\\ \vspace{0.3cm}
\begin{tikzpicture}[scale=1.2]
\tikzstyle{sommetNO}=[shape=circle,fill=white,draw=black,minimum size=0.5pt, inner sep=1.5pt]
\fill [color=gray!40] (0,0) -- (50:1) -- (10:1) -- cycle;
\fill [color=gray!40] (0,0) -- (130:1) -- (170:1) -- cycle;
\fill [color=gray!40] (0,0) -- (250:1) -- (290:1) -- cycle;
\node[draw=white, sommetNO] (X2) at (130:1) {};
\node[draw=white, sommetNO] (X3) at (170:1) {};
\node[draw=white, sommetNO] (X4) at (250:1) {};
\node[draw=white, sommetNO] (X5) at (290:1) {};
\draw [ultra thick] (0,0) -- (X2) ;
\draw [ultra thick] (0,0) -- (X3) ;
\draw [ultra thick] (X2) -- (X3) ;
\draw [ultra thin](0,0) -- (X4) ;
\draw  [ultra thin] (0,0) -- (X5) ;
\draw  [ultra thick] (X4) -- (X5) ;
\draw [ultra thick] (0,0) -- (10: 1) ;
\draw [ultra thick] (0,0) -- (50: 1) ;
\draw [ultra thick] (10:1) -- (50: 1) ;
\draw (0,0) node {$\bullet$};
\draw (10: 1) node {$\bullet$};
\draw (50: 1) node {$\bullet$};
\draw (X2) node[left] {\small{$x_2$}};
\draw (X3) node[left] {\small{$x_3$}};
\draw (X4) node[left] {\small{$x_4$}};
\draw (X5) node[right] {\small{$x_5$}};
\draw (10: 1) node[right] {\small{$x_6$}};
\draw (50: 1) node[right] {\small{$x_7$}};
\node at (90:0.3) {$x_1$};
\draw (180: 1.8) node {$\tilde{\alpha}_{2,5}$};
\end{tikzpicture} 
&
\begin{tikzpicture}[scale=1.2]
\tikzstyle{sommetNO}=[shape=circle,fill=white,draw=black,minimum size=0.5pt, inner sep=1.5pt]
\fill [color=gray!40] (0,0) -- (50:1) -- (10:1) -- cycle;
\fill [color=gray!40] (0,0) -- (130:1) -- (170:1) -- cycle;
\fill [color=gray!40] (0,0) -- (250:1) -- (290:1) -- cycle;
\node[draw=white, sommetNO] (X2) at (90:0.7) {};
\node[draw=white, sommetNO] (X3) at (170:1) {};
\node[draw=white, sommetNO] (X6) at (290:1) {};
\node[draw=white, sommetNO] (X5) at (250:1) {};
\node[draw=white, sommetNO] (X4) at (130:1) {};
\draw [ultra thick] (0,0) -- (10:1) ;
\draw [ultra thick] (0,0) -- (50:1) ;
\draw [ultra thick] (10:1) -- (50:1) ;
\draw [ultra thin](0,0) -- (X6) ;
\draw  [ultra thin] (0,0) -- (X5) ;
\draw  [ultra thick] (X5) -- (X6) ;
\draw [ultra thick] (0,0) -- (X3) ;
\draw [ultra thick] (0,0) -- (X4) ;
\draw [ultra thick] (X4) -- (X3) ;
\draw (0,0) node {$\bullet$};
\draw (10: 1) node {$\bullet$};
\draw (50: 1) node {$\bullet$};
\draw (X2) node[above] {\small{$x_2$}};
\draw (X3) node[left] {\small{$x_3$}};
\draw (X4) node[left] {\small{$x_4$}};
\draw (X5) node[left] {\small{$x_5$}};
\draw (X6) node[right] {\small{$x_6$}};
\draw (10: 1) node[right] {\small{$x_7$}};
\draw (50: 1) node[right] {\small{$x_8$}};
\draw (180: 1.8) node {$\tilde{\alpha}_{2,6}$};
\node at (90:0.3) {$x_1$};
\end{tikzpicture} 
& 
\begin{tikzpicture}[scale=1.2]
\tikzstyle{sommetNO}=[shape=circle,fill=white,draw=black,minimum size=0.5pt, inner sep=1.5pt]
\fill [color=gray!40] (0,0) -- (20:1) -- (340:1) -- cycle;
\fill [color=gray!40] (0,0) -- (160:1) -- (200:1) -- cycle;
\fill [color=gray!40] (0,0) -- (250:1) -- (290:1) -- cycle;
\fill [color=gray!40] (0,0) -- (110:1) -- (70:1) -- cycle;
\node[draw=white, sommetNO] (X7) at (250:1) {};
\node[draw=white, sommetNO] (X6) at (290:1) {};
\node[draw=white, sommetNO] (X3) at (70:1) {};
\node[draw=white, sommetNO] (X2) at (110:1) {};
\node[draw=white, sommetNO] (X4) at (160:1) {};
\node[draw=white, sommetNO] (X5) at (200:1) {};
\draw [ultra thin](0,0) -- (X2) ;
\draw  [ultra thin] (0,0) -- (X3) ;
\draw  [ultra thick] (X2) -- (X3) ;
\draw [ultra thin](0,0) -- (X4) ;
\draw  [ultra thin] (0,0) -- (X5) ;
\draw  [ultra thick] (X4) -- (X5) ;
\draw [ultra thick] (0,0) -- (20: 1) ;
\draw [ultra thick] (0,0) -- (340: 1) ;
\draw [ultra thick] (20:1) -- (340: 1) ;
\draw (0,0) node {$\bullet$};
\draw (20: 1) node {$\bullet$};
\draw (340: 1) node {$\bullet$};
\draw [ultra thick](0,0) -- (X6) ;
\draw  [ultra thick] (0,0) -- (X7) ;
\draw  [ultra thick] (X6) -- (X7) ;
\draw (0: 0) node[above right ] {\small{$x_1$}};
\draw (X2) node[left] {\small{$x_3$}};
\draw (X3) node[right] {\small{$x_2$}};
\draw (X4) node[left] {\small{$x_4$}};
\draw (X5) node[left] {\small{$x_5$}};
\draw (X6) node[right] {\small{$x_7$}};
\draw (X7) node[left] {\small{$x_6$}};
\draw (340:1) node[right] {\small{$x_8$}};
\draw (20:1) node[right] {\small{$x_9$}};

\draw (180: 1.8) node {$\tilde{\alpha}_{3,6}$};
\end{tikzpicture} 
\end{tabular}

The coordinate formulas follow directly from the above graphs.

\item

We define the forms $\tilde{\alpha}_{p,q}$ with arbitrary $p>0$ and $q>0$, except for $(p,q)=(1,4)$ and $(p,q)=(4,1)$ in the last equation, using the following rules:
\begin{eqnarray*}
\tilde{\alpha}_{q,p} &:=& \tilde{\alpha}_{p,q};\\
\tilde{\alpha}_{p,q+4} &:=& \tilde{\alpha}_{p+4,q};\\
\tilde{\alpha}_{p+2,q+2} (x_1, \ldots , x_{n+4}) & :=&  \tilde{\alpha}_{p,q} (x_1, \ldots , x_{n}) +  \tilde{\alpha}_{2,3} (x_1, x_{n+1} ,  x_{n+2},x_{n+3}, x_{n+4}) + x_1 . 
\end{eqnarray*}

\end{enumerate}

\end{df}

It is easy to check that the first eleven forms suffice to determine the rest.

\subsection{An equivalent formulation of the main result} \label{GRAPTHM}

Let us give a different way to formulate our main result.

\begin{thm}
\label{GrThm}
The cubic forms $\alpha_{p,q}$ and $\tilde{\alpha}_{p,q}$ are equivalent
for all $p,q$.
\end{thm}

\noindent
Theorems \ref{thmcomplex} and \ref{thmreal} will follow from Theorem~\ref{GrThm}
since the forms $\tilde{\alpha}_{p,q}$ have the required periodicity.

\section{Proof of Theorem~\ref{GrThm}} \label{subsectionisomorphisms}

In this section, we give explicitly step by step the
coordinate transformations that intertwine the cubic forms~$\alpha_{p,q}$ and $\tilde{\alpha}_{p,q}$. 
According to the number of generators modulo $4$, different cases appear. 
The cases where the number of generators is even will be deduced 
from cases where the number of generators is odd. 
This is explained in Sections \ref{subsection4k} and \ref{subsection4k+2}. 
In Section \ref{n=4k+3andn=4k+1}, we focus on the two cases with odd number of variables. 
Finally, we finish the proof of the Theorem~\ref{GrThm} in Section \ref{conclusion}. 

\subsection{The case $(p,q)$ with $n=4k$}
\label{subsection4k}
This first lemma shows that the case $n=4k$ can be deduced from the case $n=4k-1$.
The cubic form $\alpha_{p,q}$ with $p+q=4k$ is equivalent to a cubic form where 
the last generator completely disappears or is only present in the linear part.

We introduce the following notation.
Consider the projection $\Z_2^n\to\Z_2^{n-1}$ defined by ``forgetting'' the last coordinate, $x_n$.
The cubic form on $\Z_2^n$ obtained by the pull-back of a cubic form $\alpha$ on $\Z_2^{n-1}$
will be denoted by $\widehat{\alpha}$.
In other words,
$$
\widehat{\alpha}(x_1,\ldots,x_n)=\alpha(x_1,\ldots,x_{n-1}).
$$

\begin{lem} 
\label{4k--4k-1}
If $n= p+q = 4k$ with $k \in \mathbb{N} \backslash\{0\} $ then, one has the equivalent forms
\begin{eqnarray*}
\label{Equiv4k-n0}
\alpha_{0,n} \simeq  \widehat{\alpha}_{0,n-1}  
\hspace{0.8cm}\mbox{and} \hspace{0.8cm}
\alpha_{n,0} \simeq  \widehat{\alpha}_{n-1,0},
\end{eqnarray*}
\begin{eqnarray*}
 \label{Equiv4k-pq}
\alpha_{p,q}  \simeq
\left\{ 
\begin{array}{ccl}
\widehat{\alpha}_{p,q-1} &  & \mbox{if    }   p, q >0 \mbox{ are even}   ,  \\[4pt]
\widehat{\alpha}_{p,q-1}  + x_n  &  & \mbox{if   } p , q\geq 1 \mbox{ are odd}   . 
\end{array} 
\right. 
\end{eqnarray*}
\end{lem}

\begin{proof}
To establish the equivalence in the cases $(n,0)$, $(0,n)$ and $(p,q)$ where $p,q \geq 1$ are odd, we choose the following coordinate transformation:
\begin{eqnarray}
\label{eqn4k}
x'_1& = &x _1  + x_{4k}, \nonumber \\
x'_2& = &x_1 + x _2 + x_{4k},\nonumber  \\
\cdots & = &   \cdots   \nonumber\\
x'_i& = &x_1 +x _i +    x_{4k},  \nonumber\\
\cdots & = &   \cdots   \nonumber\\
x'_{2k+1}& = &x_1 + x_{2k+1} +  x_{4k},  \\
x'_{2k+2}& = & x_2 + \ldots + x_{2k+1} +  \widehat{x_{2k+2}} + \ldots +   x_{4k}, \nonumber \\
\cdots & = & \cdots  \nonumber \\
x'_{i}& = & x_2 + \ldots + x_{i-1} +  \widehat{x_{i}} +  x_{i+1}+ \ldots + x_{4k},\nonumber  \\
\cdots & = & \cdots  \nonumber\\
x'_{4k-1}& = &x _2 +\ldots + x_{4k-2} +  \widehat{x_{4k-1}}  +  x_{4k}, \nonumber \\
x'_{4k}& = &x _{4k} , \nonumber
\end{eqnarray}
where $\widehat{\cdot}$ denotes removed terms. 

In the case where $p,q> 0$ are even, 
it suffices to consider the particular case $p\leq q$. 
Indeed, it was proved in \cite{KMG2014} that $\alpha_{p,q} \simeq \alpha_{q,p}$ for every  $p,q> 0$, 
and that $\alpha_{q-1,p} \simeq \alpha_{q,p-1}$ for every  $p,q> 0$ if $p$ and $q$ are even and $p+q = 4k +3$. 
If $0<p \leq q$ are even, the coordinate transformation considered in this case is the same 
as (\ref{eqn4k}) for every elements $x'_i$ where $i \in \{ 2, \ldots, n-1 \} $ together with
\[
\begin{array}{ccl}
x'_1& = &  x_{4k},  \\
x'_{4k}& = &x_1 + x _{4k} . 
\end{array}
\]
Hence the lemma.
\end{proof}

\subsection{The case $(p,q)$ with $n=4k+2$ }
\label{subsection4k+2}

This second lemma allows us to reduce the case $n=4k+2$ to the case $n=4k+1$. 
The cubic form $\alpha_{p,q}$ with $p+q=4k+2$ is equivalent to a cubic form where the last coordinate is only present in the quadratic part and sometimes in the linear part. 

\begin{lem} \label{4k+2--4k+1}
If $n= p+q = 4k+2$ with $k \in \mathbb{N} \backslash\{0\} $ then, one has the equivalent forms
\begin{eqnarray}
\label{Equiv4k+2-0n}
\alpha_{0,n}  \simeq  \widehat{\alpha}_{0,n-1}   + x_n + x_n 
 \sum_{ i=1}^{  n-1} x_i  ,\\
 \label{Equiv4k+2-n0}
\alpha_{n,0}   \simeq  \widehat{\alpha}_{n-1,0}   + x_n + x_n 
  \sum_{i=1}^{  n-1}  x_i ,
\end{eqnarray}
\begin{eqnarray}
 \label{Equiv4k+2-pq}
 \displaystyle
\alpha_{p,q}    \simeq
\left\{ 
\begin{array}{lcl}
\displaystyle
\widehat{\alpha}_{p-1,q}  + x_n 
  \sum_{i = 1}^{n-1} x_i
 &  & \mbox{if    }   p , q\geq 1 \mbox{ are odd}  ,  \\[6pt]
 \displaystyle
\widehat{\alpha}_{p-1,q} + x_n + x_n 
  \sum_{i = 1}^{n-1} x_i
  &  & \mbox{if   }   p, q >0 \mbox{ are even}    . 
\end{array}
\right.
\end{eqnarray} 
\end{lem}

\begin{proof}
To establish the equivalence, we choose the following coordinate transformation:
\begin{eqnarray}
\label{eqn4k+2}
x'_1& = & x_{4k+2}, \nonumber  \\
x'_2& = &x_1+ x_{4k+2}, \nonumber  \\
\cdots & = &   \cdots  \nonumber \\
x'_i& = &x_{i-1} +    x_{4k+2},  \nonumber \\
\cdots & = &   \cdots \nonumber  \\
x'_{2k+1}& = & x_{2k} +  x_{4k+2},  \\
x'_{2k+2}& = & x_1 + \ldots + x_{4k} +  \widehat{x_{4k+1}} +   x_{4k+2},\nonumber   \\
\cdots & = & \cdots \nonumber  \\
x'_{i}& = & x_2 + \ldots + x_{6k+2-i} +  \widehat{x_{6k+3-i}} +  x_{6k+4-i}+ \ldots + x_{4k+2},  \nonumber \\
\cdots & = & \cdots  \nonumber   \\
x'_{4k+2}& = &x _2 +\ldots + x_{2k} +  \widehat{x_{2k+1}} +  x_{2k+2} + \ldots +   x_{4k+2}.    \nonumber 
\end{eqnarray}

If the signature is $(n,0)$ and $(0,n)$, the equivalence is given by the above transformation. 
If the signature is $(p,q)$ with $p,q>0$, the transformation (\ref{eqn4k+2}) is taking into account only in the case 
$p\leq q$. 
The case $q < p$ is then deduced from the case $p\leq q$ since it was proven in \cite{KMG2014} that 
\begin{align*}
\alpha_{p,q} & \simeq \alpha_{q,p} & \mbox{if   }& p,q> 0 ,\\
\alpha_{p-1,q} & \simeq \alpha_{p,q-1} & \mbox{if   }  &p+q= 4k+1  \mbox{ and } p-1 \mbox{ is even,}\\
\alpha_{q,p-1} & \simeq \alpha_{q-1,p} & \mbox{if   }  &p+q= 4k+1  \mbox{ and } q \mbox{ is odd.}
\end{align*}

Lemma \ref{4k+2--4k+1} is proved.
\end{proof}

\subsection{The case $n=4k+3$ and $n=4k+1$  }
\label{n=4k+3andn=4k+1}

If $n=p+q$ is odd, there exist exactly four distinct algebras $\bbO_{p,q}$ with $p,q \geq 0$ 
up to graded isomorphism; see \cite{KMG2014}. 
We will treat each of these four cases independently.

\begin{lem}
\label{lem4k+3}
If $n=4k+3$ and $k$ is odd, 
then
\begin{enumerate}
\item
the form $\alpha_{0,n}$ is equivalent to 
\[
  x_1  +(x_1+1) \left(  \alpha^{\mathrm{Cl}}_{2k+2,0} (x_2, \ldots , x_{2k+3}) + \alpha^{\mathrm{Cl}}_{0,2k}(x_{2k+4}, \ldots , x_{4k+3})   \right) ,
\]
\item
the form $\alpha_{n,0}$ is equivalent to 
\[
(x_1+1) \left(  \alpha^{\mathrm{Cl}}_{2k+2,0} (x_2, \ldots , x_{2k+3}) + \alpha^{\mathrm{Cl}}_{0,2k}(x_{2k+4}, \ldots , x_{4k+3})   \right) ,
\]
\item
the form $\alpha_{2k+1,2k+2}$ is equivalent to 
\[
  x_1  +(x_1+1) \left(  \alpha^{\mathrm{Cl}}_{2k+2,0}  (x_2, \ldots , x_{2k+3})  + \alpha^{\mathrm{Cl}}_{0,2k} (x_{2k+4}, \ldots , x_{4k+3})   \right) + \sum_{i=2}^{2k+3} x_i ,
\]
\item
the form $\alpha_{2k,2k+3}$ is equivalent to 
\[
  x_1  +(x_1+1) \left(  \alpha^{\mathrm{Cl}}_{2k+2,0}  (x_2, \ldots , x_{2k+3})  + \alpha^{\mathrm{Cl}}_{0,2k} (x_{2k+4}, \ldots , x_{4k+3})   \right) + \sum_{i=2}^{2k} x_i + x_{2k+3} .
\]
\end{enumerate}
If $n=4k+3$ and $k$ is even, then 

\begin{enumerate}
\item
the form $\alpha_{0,n}$ is equivalent to 
\[
  x_1 +(x_1+1) 
  \left(  \alpha^{\mathrm{Cl}}_{2k,0} (x_2, \ldots , x_{2k+1}) + \alpha^{\mathrm{Cl}}_{0,2k+2} (x_{2k+2}, \ldots , x_{4k+3})  \right) ,
\]
\item
the form $\alpha_{n,0}$ is equivalent to 
\[
(x_1+1) 
  \left(  \alpha^{\mathrm{Cl}}_{2k,0} (x_2, \ldots , x_{2k+1}) + \alpha^{\mathrm{Cl}}_{0,2k+2} (x_{2k+2}, \ldots , x_{4k+3})  \right) ,
\]
\item
the form $\alpha_{2k+1,2k+2}$ is equivalent to 
\[
  x_1  +(x_1+1) \left(  \alpha^{\mathrm{Cl}}_{2k,0}  (x_2, \ldots , x_{2k+1})  + \alpha^{\mathrm{Cl}}_{0,2k+2} (x_{2k+2}, \ldots , x_{4k+3})   \right) + \sum_{i=2k+2}^{4k+3} x_i ,
\]
\item
the form $\alpha_{2k,2k+3}$ is equivalent to 
\[
  x_1  +(x_1+1) \left(  \alpha^{\mathrm{Cl}}_{2k,0}  (x_2, \ldots , x_{2k+1})  + \alpha^{\mathrm{Cl}}_{0,2k+2} (x_{2k+2}, \ldots , x_{4k+3})   \right) + \sum_{i=2k+3}^{4k+1} x_i + x_{4k+3} . 
\]
\end{enumerate}

\end{lem}

\begin{proof}

If $k$ is odd, chose the following coordinate transformation:
\begin{eqnarray}
\label{eqn4k+3odd}
x'_1& = &x_{2k+2} + x_{2k +3},  \nonumber \\
x'_2& = &x _{2} +  x_{2k+2},  \nonumber  \\
\cdots & = &   \cdots \nonumber  \\
x'_i& = &x _{i} +  x_{2k+2}, \nonumber   \\
\cdots & = &   \cdots  \nonumber   \\
x'_{2k+1}& = &x _{2k+1} +  x_{2k+2} , \nonumber   \\
x'_{2k+2}& = &x _{2k+2} , \nonumber   \\
x'_{2k+3}& = &x _1 + \ldots  + x_{2k+1} + x_{2k+3} + \ldots + x_{4k+1}+\widehat{x_{4k+2}}  + x_{4k+3},    \nonumber   \\
\cdots & = & \cdots   \nonumber  \\
x'_{i}& = &x _1 + \ldots  + x_{2k+1} + x_{2k+3} + \ldots + x_{6k+4-i}+\widehat{x_{6k+5-i}}  + x_{6k+6-i}+\ldots + x_{4k+3},   \nonumber   \\
\cdots & = & \cdots \nonumber    \\
x'_{4k+1}& = &x _1 + \ldots  + x_{2k+1} +   x_{2k+3} + \widehat{x_{2k+4}} + x_{2k+5}+   \ldots + x_{4k+3} , \nonumber   \\
x'_{4k+2}& = &x _1 + \ldots  + x_{2k+1}+   x_{2k+3} + \ldots + x_{4k+2} ,\nonumber  \\
x'_{4k+3}& = &x _1 + \ldots + x_{2k+1} +   x_{2k+3} + \ldots + x_{4k+3} .  \nonumber  \\
\nonumber 
\end{eqnarray} 

If $k$ is even and $k \geq 2$, chose the following coordinate transformation:
\begin{eqnarray}
\label{eqn4k+3even}
x'_1& = &x_{2k+2} + x_{4k + 3},   \nonumber   \\
x'_2& = &x _{2k+2} +  x_{2k+3},    \nonumber  \\
\cdots & = &   \cdots   \nonumber  \\
x'_i& = &x _{2k+2} +  x_{2k+1+i},\nonumber    \\
\cdots & = & \cdots   \nonumber   \\
x'_{2k+1}& = &x _{2k+2} +  x_{4k+2}, \nonumber   \\
x'_{2k+2}& = &x _{2k+2} ,  \nonumber  \\
x'_{2k+3}& = &x _1 + \widehat{x_2}+ x_3 \ldots  + x_{2k+1} + x_{2k+3} + \ldots + x_{4k+3},    \nonumber   \\
\cdots & = & \cdots \nonumber   \\
x'_{i}& = &x _1 + \ldots  + x_{i-2k-2} +\widehat{x_{i-2k-1}} + x_{i-2k} +   \ldots + x_{2k+1} +x_{2k+3}  + \ldots +  x_{4k+3},\nonumber   \\
\cdots & = & \cdots  \nonumber \\
x'_{4k+2}& = &x _1 + \ldots  + x_{2k} + \widehat{x_{2k+1}} + x_{2k+3}+   \ldots + x_{4k+3} ,\nonumber  \\
x'_{4k+3}& = &x _1 + \ldots + x_{2k+1} +   x_{2k+3}  + \ldots + x_{4k+3}  . \nonumber \\
\nonumber 
\end{eqnarray} 
These coordinate transformation provide the desired equivalence. 
\end{proof}

\begin{lem}
\label{lem4k+1}
If $n=4k+1$, then 
\begin{enumerate}
\item 
the form $\alpha_{0,n}$ is equivalent to 
\[
  x_1 +(x_1+1)\left(  \alpha^{\mathrm{Cl}}_{2k,0}  (x_2, \ldots , x_{2k+1})  + \alpha^{\mathrm{Cl}}_{0,2k} (x_{2k+2}, \ldots , x_{4k+1})   \right),
  \]
  \item
the form $\alpha_{n,0}$ is equivalent to 
\[
  (x_1+1)\left(  \alpha^{\mathrm{Cl}}_{2k,0}  (x_2, \ldots , x_{2k+1})  + \alpha^{\mathrm{Cl}}_{0,2k} (x_{2k+2}, \ldots , x_{4k+1})   \right),
  \]
\item
the form $\alpha_{2k,2k+1}$ is equivalent to 
\[
  x_1 +(x_1+1)\left(  \alpha^{\mathrm{Cl}}_{2k,0}  (x_2, \ldots , x_{2k+1})  + \alpha^{\mathrm{Cl}}_{0,2k} (x_{2k+2}, \ldots , x_{4k+1})   \right) + \sum_{i=2k+2}^{4k+1}
x_i,  \]
\item
If $n=4k+1$ and $k$ is even, then $\alpha_{2k-2,2k+3}$ is equivalent to 
\[
  x_1 +(x_1+1)\left(  \alpha^{\mathrm{Cl}}_{2k,0}  (x_2, \ldots , x_{2k+1})  + \alpha^{\mathrm{Cl}}_{0,2k} (x_{2k+2}, \ldots , x_{4k+1})   \right) + \sum_{i=2k+2}^{4k+1}
x_i,   \]
  If $n=4k+1$ and $k$ is odd, then $\alpha_{2k-2,2k+3}$ is equivalent to 
\[
  x_1 +(x_1+1)\left(  \alpha^{\mathrm{Cl}}_{2k,0}  (x_2, \ldots , x_{2k+1})  + \alpha^{\mathrm{Cl}}_{0,2k} (x_{2k+2}, \ldots , x_{4k+1})   \right) + \sum_{i=2}^{2k-2} 
x_i + x_{2k+1}.   \]
  \end{enumerate}
\end{lem}

\begin{proof}
If $k\geq2$ is odd, the coordinate transformation is as follows:
\begin{eqnarray}
\label{eqn4k+1odd}
x'_1& = &x _1 + x_{2k + 1},    \nonumber \\
x'_2& = &x _1 + x_2    \nonumber     \\
\cdots & = &   \cdots    \nonumber    \\
x'_i& = &x _1 +  x_i,   \nonumber     \\
\cdots & = &   \cdots   \nonumber    \\
x'_{2k}& = &x _{1} +  x_{2k},  \\
x'_{2k+1}& = &x _1 + \ldots +   x_{4k+1},   \nonumber    \\
x'_{2k+2}& = &x _2 + \ldots + x_{4k-1} +\widehat{x_{4k}} + x_{4k+1},     \nonumber   \\
\cdots & = & \cdots      \nonumber     \\
x'_i& = &x _2 + \ldots +  x_{6k-i+1} + \widehat{x_{6k-i+2}} +   x_{6k-i+3}+\ldots + x_{4k+1}  ,        \nonumber     \\
\cdots & = & \cdots  \nonumber     \\
x'_{4k}& = &x _2 + \ldots + x_{2k+1} +   \widehat{x_{2k+2}} + x_{2k+3}  + \ldots + x_{4k+1} ,  \nonumber      \\
x'_{4k+1}& = &x _2 +  \ldots + x_{4k}  .    \nonumber   
\end{eqnarray}
If $k\geq2$ is even, the coordinate transformation is as follows:
\begin{eqnarray}
\label{eqn4k+1even}
x'_1& = &x _1 + x_{4k + 1}, \nonumber\\
\cdots & = &   \cdots    \nonumber   \\
x'_i& = &x _1 +  x_{4k +2 - i},    \nonumber  \\
\cdots & = &   \cdots   \nonumber  \\
x'_{2k}& = &x _{1} +  x_{2k+2},  \nonumber    \\
x'_{2k+1}& = &x _1 + \ldots +   x_{4k+1},  \\
x'_{2k+2}& =&  \widehat{x_2}   +x_3+  \ldots + x_{4k+1},   \nonumber  \\
\cdots & = & \cdots   \nonumber  \\
x'_i& = &x _2 + \ldots  + x_{i- 2k-1} + \widehat{x_{i -2k}}+ x_{i -2k+1} + \ldots + x_{4k+1},      \nonumber   \\
\cdots & = & \cdots    \nonumber  \\
x'_{4k+1}& = &x _2 + \ldots + x_{2k} +   \widehat{x_{2k+1}} + x_{2k+2}  + \ldots + x_{4k+1}.  \nonumber
\end{eqnarray}

The above transformations provide the equivalence in the cases where
the signature is $(0,n)$, $(n,0)$ and $(p,q)$ with $k$ odd. 

Furthermore,
in the case of signature $(2k,2k+1)$ with even $k$
(which is equivalent to the case of signature $(2k+1,2k)$),
the transformation (\ref{eqn4k+1even}) 
has to be combined with the following one
\[\begin{array}{ccl}
x'_1& = &x_1,  \\
x'_2& = &x _2 +  x_1,  \\
x'_i& = &x _{i}, \hspace{1.5cm} \mbox{if }   \hspace{0.2cm} 3 \leq  i \leq 4k+1 .
\end{array} 
\]

Finally, 
if $k$ is even and in the case of signature $(2k-2,2k+3)$,
the transformation (\ref{eqn4k+1odd})
has to be combined with the following change of coordinates 
\[\begin{array}{ccl}
x'_i& = &x _{i}, \hspace{1.5cm} \mbox{if }   \hspace{0.2cm} 1 \leq  i \leq 4k,\\
x'_{4k+1}& = &x _{4k+1} +  x_1.  
\end{array} 
\]
The results follow directly form these coordinate transformations. 
\end{proof}

The changes of coordinates (\ref{eqn4k+1odd}) and (\ref{eqn4k+1even}) 
used in the proof of the lemma \ref{lem4k+1}, can be used for the case $p+q = 4k+2=n$ 
where the last coordinate $x_{n}$ remains unchanged.

\subsection{The end of the proof}
\label{conclusion}

Consider some more properties on the quadratic form of the Clifford algebras. 
Denote, as above, by $\alpha^{\mathrm{Cl}}_{0,2}$ the generating quadratic form of the Clifford algebra 
$\mathrm{Cl}_{0,2}$ and by $\alpha^{\mathrm{Cl}}_{2,0}$ the generating quadratic form of the Clifford algebra $\mathrm{Cl}_{2,0}$. Denote also
\[
(\alpha_{2,0}^{\mathrm{Cl}})^{l} (x_1, \ldots , x_{2l}) := \alpha_{2,0}^{\mathrm{Cl}}(x_1,x_2)+ \ldots +  \alpha_{2,0}^{\mathrm{Cl}}(x_{2l-1},x_{2l}),
\]
\[
(\alpha_{0,2}^{\mathrm{Cl}})^{l} (x_1, \ldots , x_{2l}) := \alpha_{0,2}^{\mathrm{Cl}}(x_1,x_2)+ \ldots +  \alpha_{0,2}^{\mathrm{Cl}}(x_{2l-1},x_{2l}).
\]

The following lemma is useful in the Clifford case.
\begin{lem}
\label{CliffordPeriodicity}
If $k$ is even, then
\begin{eqnarray}
\label{Cliffordeven}
\alpha^{\mathrm{Cl}}_{2k,0} \simeq \alpha^{\mathrm{Cl}}_{0,2k}   \simeq( \alpha^{\mathrm{Cl}}_{0,2}  )^{k/2} + ( \alpha^{\mathrm{Cl}}_{2,0} )^{k/2} \nonumber.
\end{eqnarray}
If $k$ is odd, then 
\begin{eqnarray}
\label{Cliffordodd}
\alpha^{\mathrm{Cl}}_{2k,0} \simeq( \alpha^{\mathrm{Cl}}_{2,0}  )^{\frac{k+1}{2}} + ( \alpha^{\mathrm{Cl}}_{0,2}  )^{\frac{k-1}{2}} , \hspace{0.7cm} \alpha^{\mathrm{Cl}}_{0,2k} \simeq( \alpha^{\mathrm{Cl}}_{2,0}  )^{\frac{k-1}{2}} + ( \alpha^{\mathrm{Cl}}_{0,2}  )^{\frac{k+1}{2}} \nonumber .
\end{eqnarray}
\end{lem}

Lemma \ref{CliffordPeriodicity} means that 
the graph of the quadratic form of a Clifford algebra with even generators, 
is equivalent to a disconnected graph consisting of components of the type 
\begin{tikzpicture}[scale=1.2]
\tikzstyle{sommetYES}=[shape=circle,fill=black,draw=black,minimum size=0.5pt, inner sep=1.5pt]
\node[draw=white, sommetYES] (X2) at (0:1) {};
\node[draw=white, sommetYES] (X1) at (0:0) {};
\draw [ultra thick] (X1) -- (X2) ;
\end{tikzpicture} 
and 
\begin{tikzpicture}[scale=1.2]
\tikzstyle{sommetNO}=[shape=circle,fill=white,draw=black,minimum size=0.5pt, inner sep=1.5pt]
\node[draw=white, sommetNO] (X2) at (0:1) {};
\node[draw=white, sommetNO] (X1) at (0:0) {};
\draw [ultra thick] (X2) -- (X1) ;
\end{tikzpicture}.

\begin{lem}
\label{lemma2,3}
The form $\tilde{\alpha}_{2,3}$ is equivalent to 
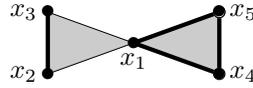
\begin{center}
\begin{tikzpicture}[scale=1.2]
\tikzstyle{sommetNO}=[shape=circle,fill=white,draw=black,minimum size=0.5pt, inner sep=1.5pt]
\fill [color=gray!40] (0,0) -- (20:1) -- (340:1) -- cycle;
\fill [color=gray!40] (0,0) -- (160:1) -- (200:1) -- cycle;
\node[draw=white, sommetNO] (X) at (20:1) {};
\node[draw=white, sommetNO] (X2) at (20:1) {};
\node[draw=white, sommetNO] (X3) at (340:1) {};
\draw [ultra thick] (0,0) -- (X2) ;
\draw [ultra thick] (0,0) -- (X3) ;
\draw [ultra thick] (X2) -- (X3) ;
\draw [ultra thin](0,0) -- (160: 1) ;
\draw  [ultra thin] (0,0) -- (200: 1) ;
\draw  [ultra thick] (160:1) -- (200: 1) ;
\draw (0,0) node {$\bullet$};
\draw (160: 1) node {$\bullet$};
\draw (200: 1) node {$\bullet$};
\draw (20: 1) node {$\bullet$};
\draw (340: 1) node {$\bullet$};
\draw (20:1) node[right] {\small{$x_5$}};
\draw (340: 1) node[right] {\small{$x_4$}};
\draw (200:1) node[left] {\small{$x_2$}};
\draw (160: 1) node[left] {\small{$x_3$}};
\draw (0:0) node[below] {\small{$x_1$}};
\end{tikzpicture} 
\end{center}
 and one has 
 \begin{eqnarray*}
\bbO_{2,3} \otimes_{\C}  \bbO_{2,3}  &\simeq &  \bbO_{1,4} \otimes_{\C}  \bbO_{0,5}  .
 \end{eqnarray*}
\end{lem}

\begin{proof}
For the first part, the coordinate transformation is given by
\begin{eqnarray}
\label{eqnO23-O32}
x'_1& = &x _1 , \nonumber\\
x'_2& = &x _3 +  x_4 + x_5,    \nonumber  \\
x'_3& = &x _2 + x_4 + x_5, \nonumber   \\
x'_4& = &x _2 + x_3 + x_5,   \nonumber   \\
x'_5& = &x _2 + x_3 + x_4.  \nonumber
\end{eqnarray}
The second part is deduced directly from the first one. 
\end{proof}

To finish the proof of the Theorem~\ref{GrThm}, we consider the four different cases. 
\begin{enumerate}

\item

The case with the signature $(0,n)$. 
Suppose that $n=4k+1$, then according to Lemma~\ref{CliffordPeriodicity},
 $\alpha_{0,n}$ is equivalent to the following form:
\[
 x_1 + (x_1+1) \left( ( \alpha_{2,0}^{\mathrm{Cl}})^{k} +  ( \alpha_{0,2}^{\mathrm{Cl}})^{k}  \right) . 
\]
If $n=4k+3$, then  $\alpha_{0,n}$  is equivalent to the following form:
\[
 x_1 + (x_1+1) \left( ( \alpha_{2,0}^{\mathrm{Cl}})^{k} +  ( \alpha_{0,2}^{\mathrm{Cl}})^{k+1}  \right) .
\]
The desired equivalence follows from
 \\
   \begin{minipage}{0.55\textwidth}
\[
x_1  + (x_1 + 1)  \left( \alpha_{2,0}^{\mathrm{Cl}} (x_2,x_3)+   \alpha_{0,2}^{\mathrm{Cl}}  (x_4,x_5)\right)  
\]
   \end{minipage}
   \begin{minipage}{0.10\textwidth}
\begin{tikzpicture}[scale=1.2]
\tikzstyle{sommetNO}=[shape=circle,fill=white,draw=black,minimum size=0.5pt, inner sep=1.5pt]
\tikzstyle{sommetYES}=[shape=circle,fill=black,draw=black,minimum size=0.5pt, inner sep=1.5pt]
\fill [color=gray!40] (0,0) -- (20:1) -- (340:1) -- cycle;
\fill [color=gray!40] (0,0) -- (160:1) -- (200:1) -- cycle;
\tikzstyle{everypath}=[color=black, line width= ultra thick]
\node[draw=white, sommetNO] (X2) at (200:1) {};
\node[draw=white, sommetNO] (X3) at (160:1) {};
\draw[ultra thin] (0:0) node [sommetYES]{} -- (200:1) node[sommetNO]{};
\draw[ultra thin] (0:0) node [sommetYES]{} -- (160:1) node[sommetNO]{};
\draw [ultra thick] (X2)  -- (X3);
\draw [ultra thick] (0,0) -- (20: 1) ;
\draw [ultra thick] (0,0) -- (340: 1) ;
\draw [ultra thick] (20:1) -- (340: 1) ;
\draw (20: 1) node {$\bullet$};
\draw (340: 1) node {$\bullet$};

\draw (180: 3) node {$\longleftrightarrow$};
\draw (20:1) node[right] {\small{$x_5$}};
\draw (340: 1) node[right] {\small{$x_4$}};
\draw (200:1) node[left] {\small{$x_2$}};
\draw (160: 1) node[left] {\small{$x_3$}};
\draw (0: 0) node[below] {\small{$x_1$}};
\draw (180: 1.8) node {$\tilde{\alpha}_{0,5}$};
\end{tikzpicture} 
   \end{minipage}

\item 
The case where the signature is $(n,0)$ directly follows from the case of signature $(0,n)$ since we have the following fact \\
   \begin{minipage}{0.55\textwidth}
\[
(x_1 + 1)  \left( \alpha_{2,0}^{\mathrm{Cl}} (x_2,x_3)+   \alpha_{0,2}^{\mathrm{Cl}}  (x_4,x_5)\right) 
\]
   \end{minipage}
   \begin{minipage}{0.15\textwidth}
\begin{tikzpicture}[scale=1.2]
\tikzstyle{sommetNO}=[shape=circle,fill=white,draw=black,minimum size=0.5pt, inner sep=1.5pt]
\tikzstyle{sommetYES}=[shape=circle,fill=black,draw=black,minimum size=0.5pt, inner sep=1.5pt]
\fill [color=gray!40] (0,0) -- (20:1) -- (340:1) -- cycle;
\fill [color=gray!40] (0,0) -- (160:1) -- (200:1) -- cycle;
\tikzstyle{everypath}=[color=black, line width= ultra thick]
\node[draw=white, sommetNO] (X2) at (200:1) {};
\node[draw=white, sommetNO] (X3) at (160:1) {};
\node[draw=white, sommetNO] (X1) at (0:0) {};

\draw[ultra thin] (X1) node [sommetNO]{} -- (200:1) node[sommetNO]{};
\draw[ultra thin] (X1) node [sommetNO]{} -- (160:1) node[sommetNO]{};
\draw [ultra thick] (X2)  -- (X3);
\draw [ultra thick] (X1) -- (20: 1) ;
\draw [ultra thick] (X1) -- (340: 1) ;
\draw [ultra thick] (20:1) -- (340: 1) ;
\draw (20: 1) node {$\bullet$};
\draw (340: 1) node {$\bullet$};

\draw (180: 3) node {$\longleftrightarrow$};
\draw (20:1) node[right] {\small{$x_5$}};
\draw (340: 1) node[right] {\small{$x_4$}};
\draw (200:1) node[left] {\small{$x_2$}};
\draw (160: 1) node[left] {\small{$x_3$}};
\draw (X1) node[below] {\small{$x_1$}};
\draw (180: 1.8) node {$\tilde{\alpha}_{5,0}$};
\end{tikzpicture} 
   \end{minipage}

\item
\label{n=4k+1(2k,2k+1)}
When the signature is $(2k,2k+1)$ with $n=4k+1$ due to 
Lemmas~\ref{lem4k+1},~\ref{CliffordPeriodicity} and~\ref{lemma2,3}, 
the form $\alpha_{2k,2k+1}$ is equivalent to the following form
\begin{eqnarray}
&&x_1  + (x_1 + 1)  \left( (\alpha_{2,0}^{\mathrm{Cl}})^{k} (x_2,\ldots, x_{2k+1})+  ( \alpha_{0,2}^{\mathrm{Cl}})^{k}  (x_{2k+2},\ldots ,x_{4k+1})\right)  + \sum_{i=2k+2}^{4k+1} x_i   \nonumber  \\
&=& (x_1 + 1)  \left( (\alpha_{2,0}^{\mathrm{Cl}})^{k-1} (x_2,\ldots, x_{2k-1})+  ( \alpha_{0,2}^{\mathrm{Cl}})^{k-1}  (x_{2k+2},\ldots ,x_{4k-1})\right)  + \sum_{i=2k+2}^{4k-1} x_i  \nonumber   \\
&& +  \tilde{\alpha}_{2,3} \nonumber (x_1, x_{2k},x_{2k+1} ,x_{4k},x_{4k+1}). 
\end{eqnarray}

The conclusion is obvious since we have the following isomorphism
\\
   \begin{minipage}{0.60\textwidth}
\[
x_1  + (x_1 + 1)  \left( \alpha_{2,0}^{\mathrm{Cl}} (x_2,x_3)+   \alpha_{0,2}^{\mathrm{Cl}}  (x_4,x_5)\right)  + x_4 + x_5
\]
   \end{minipage}
   \begin{minipage}{0.15\textwidth}
\begin{tikzpicture}[scale=1.2]
\tikzstyle{sommetNO}=[shape=circle,fill=white,draw=black,minimum size=0.5pt, inner sep=1.5pt]
\tikzstyle{sommetYES}=[shape=circle,fill=black,draw=black,minimum size=0.5pt, inner sep=1.5pt]
\fill [color=gray!40] (0,0) -- (20:1) -- (340:1) -- cycle;
\fill [color=gray!40] (0,0) -- (160:1) -- (200:1) -- cycle;
\tikzstyle{everypath}=[color=black, line width= ultra thick]
\node[draw=white, sommetNO] (X2) at (200:1) {};
\node[draw=white, sommetNO] (X3) at (160:1) {};
\node[draw=white, sommetNO] (X1) at (0:0) {};
\node[draw=white, sommetNO] (X5) at (20:1) {};
\node[draw=white, sommetNO] (X4) at (340:1) {};

\draw[ultra thin] (X1) node [sommetNO]{} -- (200:1) node[sommetNO]{};
\draw[ultra thin] (X1) node [sommetNO]{} -- (160:1) node[sommetNO]{};

\draw [ultra thick] (X2) -- (X3);
\draw [ultra thick] (X1) -- (X5) ;
\draw [ultra thick] (X1) -- (X4) ;
\draw [ultra thick] (X5) -- (X4) ;

\draw (0: 0) node {$\bullet$};

\draw (180: 2.8) node {$\longleftrightarrow$};
\draw (20:1) node[right] {\small{$x_5$}};
\draw (340: 1) node[right] {\small{$x_4$}};
\draw (200:1) node[left] {\small{$x_2$}};
\draw (160: 1) node[left] {\small{$x_3$}};
\draw (X1) node[below] {\small{$x_1$}};
\draw (180: 1.8) node {$\tilde{\alpha}_{2,3}$};
\end{tikzpicture} 

   \end{minipage}

When the signature is $(2k+1,2k+2)$ with $n=4k+3$, also due to 
Lemmas~\ref{lem4k+3},~\ref{CliffordPeriodicity} and~\ref{lemma2,3}, 
the form $\alpha_{2k+1,2k+2}$ is equivalent to the following form
\begin{eqnarray}
&&
x_1  + (x_1 + 1)  \left( (\alpha_{2,0}^{\mathrm{Cl}})^{k} (x_2,\ldots, x_{2k+1})+  ( \alpha_{0,2}^{\mathrm{Cl}})^{k+1}  (x_{2k+2},\ldots ,x_{4k+3})\right)  + \sum_{i=2k+2}^{4k+3} x_i
  \nonumber  \\
&=& (x_1 + 1)  \left( (\alpha_{2,0}^{\mathrm{Cl}})^{k} (x_2,\ldots, x_{2k+1})+  ( \alpha_{0,2}^{\mathrm{Cl}})^{k}  (x_{2k+2},\ldots ,x_{4k+1})\right)  + \sum_{i=2k+2}^{4k+1} x_i  \nonumber   \\
&& +  \tilde{\alpha}_{1,2} \nonumber (x_1,x_{4k+2},x_{4k+3}).
\end{eqnarray}

\item When the signature is $(2k,2k+3)$ with $n=4k+3$, the form $\alpha_{2k,2k+3}$ is equivalent to the following form
\begin{eqnarray}
&&x_1  + (x_1 + 1)  \left( (\alpha_{2,0}^{\mathrm{Cl}})^{k} (x_2,\ldots, x_{2k+1})+  ( \alpha_{0,2}^{\mathrm{Cl}})^{k+1}  (x_{2k+2},\ldots ,x_{4k+3})\right)  + \sum_{i=2k+2}^{4k+1} x_i \nonumber  \\
&=&(x_1 + 1)  \left( (\alpha_{2,0}^{\mathrm{Cl}})^{k-1} (x_2,\ldots, x_{2k-1})+  ( \alpha_{0,2}^{\mathrm{Cl}})^{k-1}  (x_{2k+2},\ldots ,x_{4k-1})\right)  + \sum_{i=2k+2}^{4k-1} x_i  \nonumber   \\
&& +  \tilde{\alpha}_{2,5} \nonumber (x_1, x_{2k},x_{2k+1} ,x_{4k},x_{4k+1} ,x_{4k+2},x_{4k+3}).
\end{eqnarray}
When the signature is $(2k-1,2k+2)$ with $n=4k+1$, the form $\alpha_{2k-1,2k+2}$ is equivalent to $\alpha_{2k-2,2k+3}$ which is equivalent to the following form
\begin{eqnarray}
&&x_1  + (x_1 + 1)  \left( (\alpha_{2,0}^{\mathrm{Cl}})^{k} (x_2,\ldots, x_{2k+1})+  ( \alpha_{0,2}^{\mathrm{Cl}})^{k+1}  (x_{2k+2},\ldots ,x_{4k+1})\right)  + 
\sum_{i=2k+2}^{4k-1} x_i \nonumber  \\
&=&(x_1 + 1)  \left( (\alpha_{2,0}^{\mathrm{Cl}})^{k-2} (x_2,\ldots, x_{2k-3})+  ( \alpha_{0,2}^{\mathrm{Cl}})^{k-2}  (x_{2k+2},\ldots ,x_{4k-3})\right)  + \sum_{i=2k+2}^{4k-3} x_i  \nonumber   \\
&& +  \tilde{\alpha}_{3,6} \nonumber (x_1, x_{2k-2},x_{2k-1},x_{2k},x_{2k+1} ,x_{4k-2},x_{4k-1}, x_{4k},x_{4k+1}).
\end{eqnarray}
\end{enumerate}

Theorem~\ref{GrThm} is proved. 

\begin{lem}
The following isomorphism holds for any $p,q >0$,

$$ \bbO_{p,q} \otimes_{\bbC}  \bbO_{2,3} \simeq  \bbO_{p,q} \otimes_{\bbR^2}  \bbO_{3,2}. $$

\end{lem}

\begin{proof}
The cubic form  
$\alpha_{p,q}(x_1,x_2,  \ldots , x_n)  +\alpha_{2,3} (x_{n+1},x_{n+2},x_{1},x_{n+3},x_{n+4}) $, where the linear term $x_1$ appears in the form,
is equivalent to 
$\alpha_{p,q}(x_1, \ldots , x_n)  +\alpha_{3,2} (x_{1},x_{n+1},x_{n+2},x_{n+3},x_{n+4}) $ where the linear $x_1$ is absent of the form. 
This is expressed 
by the following change of coordinates
\[
\begin{array}{ccl}
x'_1& = &x _1, \\
x'_2& = &x _5,  \\
x'_3& = &x _1 + x_2  + x_3+ x_5,  \\
x'_4& = &x _1 + x_2+ x_4 + x_5, \\
x'_5& = &x_2.
\end{array}
\]
This coordinate transformation does not change the coordinates of the cubic form~$\alpha_{p,q}$.  
\end{proof}

This provides the last isomorphism of Theorem \ref{thmreal}.

\medskip

\noindent \textbf{Acknowledgments}.
The author is grateful to Valentin Ovsienko and Sophie Morier-Genoud for stimulating discussions. She also thanks the Institute Camille Jordan, CNRS, Universit\'e Lyon 1  for hospitality.
This research was partially supported by the Interuniversity Attraction Poles Programme initiated by the Belgian Science Policy Office. 
\vskip 1cm


\end{document}